\documentclass[11pt, reqno]{amsart}
\usepackage{lmodern}
\usepackage{amsmath, amsthm, amssymb, amsfonts}
\usepackage[normalem]{ulem}
\usepackage{hyperref}
\usepackage[all,cmtip]{xy}
\usepackage{verbatim}
\usepackage{nccmath}
\usepackage{stmaryrd}
\usepackage{caption}
\usepackage{mathrsfs}
\usepackage{mathtools}
\usepackage{esvect}
\usepackage{cite}
\usepackage{bbm}
\usepackage{eucal}

\usepackage{mathbbol}
\usepackage{tabularx}
\usepackage[toc,page]{appendix}
\usepackage{tikz-cd}

\usepackage{array}

\usepackage{hyperref}

\theoremstyle{theorem}
\newtheorem{theorem}{Theorem}
\newtheorem{corollary}{Corollary}
\newtheorem{lemma}{Lemma}
\newtheorem{proposition}{Proposition}

\theoremstyle{definition}
\newtheorem{definition}{Definition}

\newtheorem{remark}{Remark}
\newtheorem{question}{Question}

\theoremstyle{remark}

\newcommand{\bz}{\mathbf{z}}

\newcommand{\BA}{{\mathbb{A}}}

\newcommand{\BC}{{\mathbb{C}}}

\newcommand{\BL}{{\mathbb{L}}}
\newcommand{\BM}{{\mathbb{M}}}
\newcommand{\BN}{{\mathbb{N}}}

\newcommand{\BQ}{{\mathbb{Q}}}
\newcommand{\BR}{{\mathbb{R}}}

\newcommand{\BT}{{\mathbb{T}}}

\newcommand{\BZ}{{\mathbb{Z}}}

\newcommand{\CB}{{\mathcal B}}
\newcommand{\CC}{{\mathcal C}}
\newcommand{\CD}{{\mathcal D}}
\newcommand{\CE}{{\mathcal E}}
\newcommand{\CF}{{\mathcal F}}

\newcommand{\CI}{{\mathcal I}}

\newcommand{\CK}{{\mathcal K}}

\newcommand{\CM}{{\mathcal M}}

\newcommand{\CO}{{\mathcal O}}
\newcommand{\CP}{{\mathcal P}}

\newcommand{\CS}{{\mathcal S}}

\newcommand{\CW}{{\mathcal W}}

\newcommand{\CY}{{\mathcal Y}}

\newcommand{\sH}{\mathsf{H}}
\newcommand{\sS}{\mathsf{S}}

\DeclareFontFamily{OT1}{rsfs}{}
\DeclareFontShape{OT1}{rsfs}{n}{it}{<-> rsfs10}{}
\DeclareMathAlphabet{\curly}{OT1}{rsfs}{n}{it}

\newcommand{\Aut}{\operatorname{Aut}}
\newcommand{\p}{\mathbb{P}}

\newcommand\Spec{\operatorname{Spec}}

\newcommand{\Mbar}{{\overline M}}

\newcommand\ev{\operatorname{ev}}

\newcommand{\Hilb}{\mathsf{Hilb}}

\newcommand{\Sym}{\mathsf{Sym}}

\newcommand{\id}{\mathrm{id}}

\newcommand{\FM}{\mathsf{FM}}

\newcommand{\ch}{\mathsf{ch}}

\newcommand{\rel}{\mathsf{rel}}

\allowdisplaybreaks
\begin{document}

	\title{On the Hilbert--Chow crepant resolution conjecture}

	\author{Denis Nesterov}

	\begin{abstract} We prove the Hilbert--Chow crepant resolution conjecture in the exceptional curve classes for all projective surfaces and all genera. In particular, this confirms Ruan's cohomological Hilbert--Chow  crepant resolution conjecture. The proof exploits Fulton--MacPherson compactifications, reducing the conjecture to the case of the affine plane. 
		

	\end{abstract}

	\maketitle
	
	\setcounter{tocdepth}{1} 
	\tableofcontents
	
	\section{Introduction}

\subsection{Overview}
Let $S$ be a smooth projective surface over the field of complex numbers $\BC$. Consider its Hilbert scheme of points $\Hilb_n(S)$ and orbifold symmetric product $\Sym_n(S)$. The Nakajima--Grojnowski \cite{Nak,Gro} basis gives rise  to an isomorphism between the  cohomology of  $\Hilb_n(S)$ and the orbifold cohomology of $\Sym_n(S)$, defined by Chen and Ruan \cite{ChenRuan}, 
\begin{equation} \label{Nakajimaident}
H^*(\Hilb_n(S), \BC)\cong H_{\mathrm{orb}}^*(\Sym_n(S),\BC).
\end{equation}
The isomorphism respects the intersection pairings; however, in general, it does not respect the ring structure.  The cohomological crepant resolution conjecture, proposed by Ruan \cite{Ruan},  predicts that the one-parameter quantum cohomology $H_q^*(\Hilb_n(S), \BC)$,  specialised at $q=-1$, provides the correction to the failure of the ring homomorphism. Here, 
\[
H_q^*(\Hilb_n(S), \BC):= H^*(\Hilb_n(S), \BC)\otimes_{\BC} \BC[q]_{(q+1)},
\] such that $\BC[q]_{(q+1)}$ is  the polynomial ring in the variable $q$ localised  at the maximal ideal $(q+1)$, that is, the ring of rational functions with no pole at $q=-1$. The quantum product is defined using the Gromov--Witten invariants associated to the curve classes in $\Hilb_n(S)$ that are contracted by the Hilbert--Chow morphism
\[
\Hilb_n(S) \rightarrow S^n/\Sigma_n .
\]
Such curve classes are referred to as \textit{exceptional}. In the present work, we  prove Ruan's conjecture. 

\begin{theorem} \label{Ruan} For all projective surfaces $S$, we have a graded isomorphism of rings,
	\[ 
	H_{q}^*(\Hilb_n(S), \BC)  \cong H_{\mathrm{orb}}^*(\Sym_n(S),\BC),
	\]
	after substituting $q=-1$. 
	\end{theorem} 

Previously, this was known for simply connected surfaces by Li and Qin \cite{LiQin}, using  the Nakajima–Grojnowski action of the Heisenberg algebras. Our approach relies on an effective wall-crossing expression of Gromov--Witten invariants of $\Hilb_n(S)$ (resp.\ $\Sym_n(S)$) in terms of Gromov--Witten invariants of  $\Hilb_n(\BC^2)$ (resp.\ $\Sym_n(\BC^2)$) and integrals on Fulton--MacPherson compactifications $\FM_n(S)$, Theorem \ref{cor}. The latter depends universally on the characteristic numbers of $S$. Hence, not only Theorem \ref{cor} allows one to prove a more general version of Theorem \ref{Ruan}, as stated  in Theorem \ref{mainresult}, but also to compute the Gromov--Witten invariants of $\Hilb_n(S)$ in terms of those of  $\Hilb_n(\BC^2)$. For example, this leads to the following result, established in \cite{TNe}, which resolves \cite[Conjecture 1.3]{HQ}.

\begin{theorem}[\hspace{-0.1cm} \cite{TNe}] \label{genus1} Let $D=-\mathrm{c}_1(\mathcal{O}^{[n]})$ be the boundary divisor of $\Hilb_n(S)$. Consider the generating series of the genus 1 Gromov--Witten invariants in exceptional curve classes,  
	\begin{align*}
		F(Q,q):=\sum_{n\geq 0} \Big( \sum_{d\geq 0} \langle D \rangle_{1,d}^{\mathsf{H}_n} q^d\Big) Q^n.
	\end{align*}
	The series $F(Q,q)$ admits the following expression,
		\[ 
		F(Q,q)=\frac{1}{24} \int_S \mathrm{c}_1(S)^2 \, P(Q)^{\int_S\mathrm{c}_2(S)-1} \, T(Q,q).
		\]
	where 
	\begin{align*} 
		P(Q):=\sum_{n\geq 0} p(n)Q^n, \quad 
		T(Q,q):=\sum_{n\geq 0} \Big( \sum_{d\geq 0}\mathrm{Tr}_{n,d} q^d   \Big) Q^n,
	\end{align*}
	such that  $p(n)$ is the number of partitions of $n$, and  $\mathrm{Tr}_{n,d}$ is the trace of the quantum multiplication by the divisor $D$ on $\Hilb_n(\BC^2)$, determined by  \cite[Theorem 1]{OkPa}. 
\end{theorem}


\subsection{Generalised crepant resolution conjecture} Theorem \ref{mainresult}, a generalisation of Theorem \ref{Ruan}, asserts the equality,  after substituting \( q = -1 \), of the Gromov--Witten invariants of \( \Hilb_n(S) \) in the exceptional curve classes with the degree-zero Gromov--Witten invariants of \( \Sym_n(S) \), for arbitrary genus and number of marked points. This, in particular, confirms a stronger version of the Hilbert--Chow crepant resolution conjecture in exceptional curve classes, proposed in \cite{BG, CIT, CRuan}.

To state the result, for $2g-2+k>0$, let 
\[
\langle \lambda_1, \dots, \lambda_k \rangle^{\sH_n}_{g}(q):=\sum_{d\geq 0}\langle \lambda_1, \dots, \lambda_k \rangle^{\sH_n}_{g,d}q^d
\]
denote the generating series of the Gromov--Witten invariants of $\Hilb_n(S)$ in the exceptional curve classes, such that  $\lambda_i \in H^*(\Hilb_n(S),\BC)$ are some cohomology classes, $g$ is the genus of curves, and $d$ is the multiplicity of an exceptional curve class.   Similarly, let 
\[
\langle \lambda_1, \dots, \lambda_k \rangle^{\sS_n}_{g}
\]
be the orbifold Gromov--Witten invariants of  $\Sym_n(S)$ in the zero class. 

\begin{theorem} \label{mainresult} Assume  $2g-2+k>0$, then for all projective surfaces $S$, the series $\langle \lambda_1, \dots, \lambda_k \rangle^{\sH_n}_{g}(q)$ is the Taylor expansion of a rational function with no pole at $q=-1$.  Moreover, we have
	\[
	 \langle \lambda_1, \dots, \lambda_k \rangle^{\sH_n}_{g}(q)=\langle \lambda_1, \dots, \lambda_k \rangle^{\sS_n}_{g}
	\]
 after applying the identification (\ref{Nakajimaident}) and substituting $q=-1$. 
\end{theorem}

By the dimension constraint, for a projective surface $S$, the invariants vanish if $g\geq 2$. The vanishing does not apply to surfaces with torus-actions. More generally, Theorem \ref{mainresult} holds on the level of Cohomological field theories and for $S$ with a torus-action, as explained in \cite{TNe}. 

\subsection{Methods}The proof of Theorem \ref{mainresult} exploits another space closely related to $\Hilb_n(S)$ and $\Sym_n(S)$, namely the unordered Fulton--MacPherson compactification,
\[
\FM_n(S),
\] 
constructed in \cite{FM}. 
All three spaces are different compactifications of configuration spaces of points on $S$. However, $\FM_n(S)$ occupies a special place among them, taking the role of a bridge between $\Hilb_n(S)$ and $\Sym_n(S)$, as both  admit interpolating spaces to $\FM_n(S)$, which depend on a real parameter $\epsilon \in \BR_{>0}$, 
\[ 
\Hilb^{\epsilon}_n(S) \quad \text{and} \quad \Sym^\epsilon_n(S).
\]
The first arises through the author's notion of   $\epsilon$-weighted subschemes \linebreak \cite{NHilb}.  
The second is provided by Hassett's stability \cite{Hass}, later generalised by Routis \cite{Rou} to arbitrary dimension.  These interpolating spaces give rise to wall-crossing formulas. In \cite{NHilb}, they were explored in the context of tautological integrals on $\Hilb_n(S)$. Here, we apply them to Gromov--Witten invariants of $\Hilb_n(S)$. Using wall-crossing formulas,  we can reduce the crepant resolution conjecture of a general surface $S$ to the corresponding conjecture for the affine plane $\BC^2$, which was established by Pandharipande and Tseng in \cite{PanT}. The quantum multiplication by the boundary divisor was determined by Okounkov and Pandharipande in \cite{OkPa}; it is sufficient to deduce Theorem \ref{mainresult} for $g=0$. 

The case $2g - 2 + k \le 0$ is special because degree $d=0$ invariants are not defined. With appropriate conventions, it is still expected to satisfy the crepant resolution conjecture; however, it was not pursued in \cite{PanT} and is therefore not pursued here.


\subsection{Exceptional curve classes} The main reason we restrict to exceptional curve classes is a somewhat technical but simple result,  Proposition \ref{Proprel}, which is best illuminated through the perspective of product threefolds $S\times C$. Recall that by the moduli interpretation of $\Hilb_n(S)$, to a map 
\[
f \colon  C \rightarrow \Hilb_n(S),
\]
we can associate a 1-dimensional subscheme 
\[
\Gamma \subset S\times C,
\]
which is flat of degree $n$ over $C$. 
The class $f_*([C])$ is exceptional if and only if $\Gamma$ is supported on the horizontal copies of $C$,
\[ 
\{ x_1, \hdots, x_\ell\}\times C \subset S\times C. 
\]
The wall-crossing considered in Section \ref{Mapsmaster} can be viewed as associated to the parameter $\epsilon \in \BR_{>0}$ that controls the length of the nilpotent thickening at $\{x_i\}\times C$  at the expense of introducing Fulton--MacPherson degenerations in at the points of  $S$. 
For $\Sym_n(S)$, the same principle holds, but instead of the nilpotent thickening of subschemes, we vary the degrees of ramified covers of horizontal curves in $S\times C$.  This heuristics is realised in terms of the relative moduli spaces of maps to $\Hilb^{\epsilon}_n(S)$ and $\Sym^{\epsilon}_n(S)$, whose properness is established in Proposition \ref{Proprel} and, for master spaces, in Proposition \ref{masterrelative}.

 Since  curves in the exceptional  classes must have at least length 2 thickening at one of its components, by setting $\epsilon=1$ we obtain an empty moduli space of maps, unless the class is zero. Similarly, for $\Sym_n(S)$, only degree $1$ covers of  horizontal curves are allowed for $\epsilon=1$.  This means that Gromov--Witten invariants of $\Hilb_n(S)$ (resp.\ $\Sym_n(S)$) can be described solely in terms of wall-crossing invariants, which involve Gromov--Witten invariants of $\Hilb_n(\BC^2)$ (resp.\ $\Sym_n(\BC^2)$) and $\psi$-integrals on $\FM_n(S)$. Theorem \ref{mainresult} follows from Theorem \ref{cor} and the crepant resolution conjecture for Hilbert schemes of points on $\BC^2$, established in \cite{PanT}. Theorem \ref{Ruan} is obtained from Theorem \ref{mainresult} by setting $g=0$ and $k=3$. 

For non-exceptional  classes, subschemes  $\Gamma \subset S \times C$ are no longer supported on horizontal curves. Hence,  Fulton--MacPherson degenerations at points $x \in S$ are not sufficient to prevent collisions of such subschemes in $S \times C$. More sophisticated wall-crossing techniques than those introduced in \cite{NHilb} and in this work are therefore required. 
\subsection{Fulton--MacPherson spaces} It is remarkable that  $\FM_n(S)$ interpolates between  $\Hilb_n(S)$ and $\Sym_n(S)$, especially in light of the fact that its cohomology is governed by completely different principles, e.g., it is far  from being isomorphic to $H^*(\Hilb_n(S),\BC)$. It nevertheless  interacts with the classical and Gromov--Witten intersection theories of $\Hilb_n(S)$ through the wall-crossing formulas established in \cite{NHilb} and the present work.  

Another connection between $\FM_n(S)$ and $\Hilb_n(S)$ is provided by their blow-up interpretations. On the one hand, in \cite{FM}, the ordered Fulton--MacPherson compactification $\FM_{[n]}(S)$ is constructed as an iterated blow-up of the strata of the diagonals 
\[
\Delta=\{(x_1,\dots,x_n) \in S^n \mid x_i=x_j \text{ for  }i\neq j \} \subset S^n,
\]
 i.e., the locus where at least two points coincide. On the other, by a result of Haiman \cite{Hai}, the isospectral Hilbert scheme\footnote{It can  be viewed as an ordered version of Hilbert schemes.} $\Hilb^{\mathrm{iso}}_n(S)$ is given by the blow-up at the whole locus $\Delta$. Using these blow-up descriptions, one can show that the natural birational map
 \[
 \pi\colon \FM_{[n]}(S) \dashrightarrow \Hilb^{\mathrm{iso}}_n(S)
 \]
 does not extend because the pullback of $\Delta$ to $\FM_{[n]}(S)$ has an embedded component. For $n=3$, the embedded component is precisely the locus of collinear points on a $\p^2$-bubble.  This raises the question whether some other space of points  resolves the indeterminancies of the map above. 
 
 \begin{question} Does there exist a moduli-theoretic resolution of $\pi$?  
 	\end{question}

We expect that the answer to the above question should involve spaces of non-collinear points.



\subsection{GW/DT correspondence} The crepant resolution conjecture is related to the Gromov--Witten/Donaldson--Thomas correspondence, proposed in \cite{MNOP1, MNOP2}, for threefolds $S\times C$ through a different kind of wall-crossing formulas established in \cite{NePi} and \cite{N22}. The results in this work therefore imply the correspondence for $S\times C$ in classes that are zero on the first factor. 

  Let 
\[
\CC_{g,k} \rightarrow \Mbar_{g,k}
\]
be the universal genus $g$ curve with $k$ markings. One can consider the Donaldson--Thomas theory of the moving threefold 
\[ 
\pi_g \colon S \times \CC_{g,k} \rightarrow \Mbar_{g,k},
\]
that is, the virtual intersection theory of relative Hilbert schemes of 1-dimensional subschemes associated to $\pi_g$. In this theory, 
the insertions are given by specifying   relative incidence conditions at the vertical divisors associated to the marked points $p_i$,
\[ 
S\times \{p_i\} \subset S\times C , \quad i\in \{1, \hdots, k\},
\]
which correspond to cohomology classes $\lambda_i \in H^*(\Hilb_n(S),\BC)$. Let 
\[
\langle \lambda_1, \dots, \lambda_k \rangle^{\mathsf{DT}_n}_{g}(q)=\sum_{d\geq 0}\langle \lambda_1, \dots, \lambda_k \rangle^{\mathsf{DT}_n}_{g,d}q^d
\]
denote that generating series of Donaldson--Thomas invariants of $\pi_g$ in the curve class
\[
\ch_2=(0,0,n) \in H^2(S) \oplus H^1(C) \otimes H^3(S) \oplus  H^4(S)=H^4(S\times C)
\]
and with the third Chern  character
\[
\ch_3=d \in H^6(S\times C).
\]  The quasimap wall-crossing formula of \cite{NePi} gives rise to the  following DT/Hilb correspondence. 

\begin{theorem}[\hspace{-.01cm}\cite{NePi}] \label{quasimapwall} Assume $2g-2+k>0$, then 
\[	
\langle \lambda_1, \dots, \lambda_k \rangle^{\mathsf{DT}_n}_{g}(q)= \langle \lambda_1, \dots, \lambda_k \rangle^{\sH_n}_{g}(q). 
\]
For $2g-2+k\leq 0$, there is an explicit correction term computed in \cite[Proposition 8.6]{NePi}.\footnote{The result in \cite[Proposition 8.6]{NePi} is stated for del Pezzo surfaces, but it holds for any surface if the class is $(0,n)$. The assumption $2g-2+k>0$ is needed to use the divisor equation for the simplification of the wall-crossing terms in \cite[Theorem C]{NePi}.}
\end{theorem}

Similarly, let
\[
\langle \lambda_1, \dots, \lambda_k \rangle^{\mathsf{GW}_n}_{g}(u)
=
\sum_{m\geq 0}
\langle \lambda_1, \dots, \lambda_k \rangle^{\mathsf{GW}_n}_{g,m}\,u^{m}.
\]
be the Gromov--Witten invariants of $\pi_g$ in the same curve class $(0,0,n)$, such that  the branching divisor of degree~$m$.
The degree of the branching divisor determines the genus of the source curves by the Riemann--Hurwitz formula, as explained in \cite[Lemma 2.6]{N22}. Then, the Gromov--Witten/Hurwitz wall-crossing formula \cite{N22} gives rise to the following GW/Sym correspondence. See \cite[Section 6.3]{N22} for how the degree of the branching divisor and the insertion of the divisor class $D$ are related. 
\begin{theorem}[\hspace{-.01cm}\cite{N22}] \label{GWwall} Assume $2g-2+k>0$, then 
\[	
\langle \lambda_1, \dots, \lambda_k \rangle^{\mathsf{GW}_n}_{g}(u)= \sum_{m\geq0} \frac{1}{m!} \langle \lambda_1, \dots, \lambda_k, \underbrace{D,\dots, D}_{m} \rangle^{\sS_n}_{g} u^m,
\]
where $D$ is the orbifold class associated to the partition $(2,1,\hdots,1)$. For $2g-2+k\leq0$, there is an explicit correction term computed in \cite[Proposition 5.1]{N22}. 
\end{theorem}
Combining Theorems \ref{mainresult}, \ref{quasimapwall}, \ref{GWwall} and the divisor equation on $\Hilb_n(S)$, we obtain a wall-crossing proof of the DT/GW correspondence for $\pi_g$ in a class $(0,0,n)$; we refer to \cite[Section 6.4]{N22} for more details.  
\begin{corollary} Assume $2g-2+k>0$, then 
	\[ 
	\langle \lambda_1, \dots, \lambda_k \rangle^{\mathsf{DT}_n}_{g}(q)=\langle \lambda_1, \dots, \lambda_k \rangle^{\mathsf{GW}_n}_{g}(u)
	\] 
	after substituting $q=-e^{iu}$. 
\end{corollary}	
We believe that these techniques extend beyond  the classes of the form $(0,0,n)$. The results of \cite{NePi,N22} indeed hold for all classes. On the other hand, a more general version of Theorem \ref{cor} is expected to express arbitrary Gromov--Witten invariants of $\Hilb_n(S)$ and $\Sym_n(S)$  in terms of Gromov--Witten invariants of $\Hilb_n(\BC^2)$ and $\Sym_n(\BC^2)$, and Gromov--Witten invariants of $\FM_n(S)$. We would even speculate that this phenomenon  arises for other (non-product) threefolds. The main evidence is the wall-crossing  between unramified and standard Gromov--Witten invariants considered in \cite{NuGW}, which, in a certain sense, generalises the wall-crossing from \cite{N22} to non-product geometries.

  Developing a theory analogous to unramified Gromov–Witten theory on the sheaf side remains an open problem. Constructing a theory  similar to the one provided by the moduli space maps to Fulton–MacPherson spaces for arbitrary threefolds is an even more challenging task which is directly related to the hypothetical  moduli space of ``embedded" subschemes/sheaves/maps whose virtual intersection theory gives Gopakumar--Vafa invariants. 
\begin{question} Can the GW/DT correspondence be proved for all threefolds by using wall-crossing techniques to reduce it to the case of local curves?
	\end{question}

The recent proof of the GW/DT correspondence for non-negative geometries was carried out using symplectic techniques \cite{Par},  reducing the correspondence to the case of local curves. Note that, in its current form, it is not applicable to the relative geometry $\pi_g$ with relative insertions. 
\subsection{Acknowledgments}
The idea that the crepant resolution conjecture could be proved by relating Hilbert schemes and symmetric products via a common intermediate space was conceived during my PhD. However, it was undermined by a negative answer to the  question, “Does there exist a (sheaf-theoretic) stability condition interpolating between the two spaces?”, posed to Thorsten Beckmann.\footnote{The coarse quotient \(S^n/\Sigma_n\) can indeed be viewed as a moduli space of rank \(1\) sheaves for the zero Gieseker stability, while \(\Hilb_n(S)\) is associated to a nonzero Gieseker stability. One could therefore speculate that \(\Sym_n(S)\) arises from a stability condition more general than Gieseker stability (for example, a Bridgeland stability condition).
} The idea remained speculative until I learned about complex Fulton–MacPherson spaces through the article of Kim, Kresch, and Oh \cite{KKO}.

  I thank Georg Oberdieck for helpful discussions on Nakajima–Grojnowski classes, and Rahul Pandharipande for insights into the results of \cite{PanT} and the argument for the rationality in Section \ref{rat}. This research was supported by a Hermann-Weyl-instructorship at the  Forschungsinstitut
f\"ur Mathematik at ETH Z\"urich.

\subsection{Conventions} We work over the field of complex numbers $\BC$. Various spaces of ordered and unordered points will feature in this article. We reserve the notation $[n]$ for the set $\{1,\dots, n\}$ and use it whenever ordered points are considered. On the other hand, $n$ denotes an integer and is used in the case of unordered points.  The cohomology and homology are taken with complex coefficients, unless stated otherwise. 
\section{Preliminaries} \label{SecPrel}
\subsection{Nakajima--Grojnowski  classes} \label{NGbasis} Let 
\[
\Hilb_n(S)
\]
denote the moduli space of 0-dimensional  subschemes of length $n$ on $S$. Choose an ordered homogenous basis  
\[
\{\gamma_1, \dots, \gamma_r \}
\]
 of $H^*(S, \BC)$. Given a class $\gamma \in H^*(S, \BC)$ and an integer $k>0$, let 
\[ 
a_{k}(\gamma) \colon H^*(\Hilb_n(S),\BC) \rightarrow H^*(\Hilb_{n+k}(S),\BC) 
\]
be the Nakajima--Grojnowski raising operator \cite{Nak, Gro}. 

Consider now a cohomologically weighted partition of $n$, 
\begin{equation} \label{part}
\mu=((\mu_1, \gamma_{i_1}),\dots, (\mu_{\ell(\mu)}, \gamma_{i_{\ell(\mu)}})), 
\end{equation}
that is, a partition  $\sum \mu_j=n$  together with  classes $\gamma_{i_j} \in \{\gamma_1, \dots, \gamma_r \}$  assigned to its parts. We put a standard order on the parts of partitions, 
\begin{equation} \label{standard}
(\mu_j, \gamma_{i_j}) \geq  (\mu_{j'}, \gamma_{i_{j'}}) \quad \text{if } \mu_j\geq \mu_{j'}\text{, or if }\mu_j=\mu_{j'}\text{ and }i_j\geq i_{j'}.
\end{equation}
 Using the operators $a_{k}(\gamma)$,  we  associate a Nakajima--Grojnowski  class to $\mu$, 
\[\lambda_{\sH}(\mu):=\frac{1}{\mathfrak{z}(\mu)} a_{\mu_1}(\gamma_{i_1})\dots a_{\mu_{\ell(\mu)}}(\gamma_{i_{\ell(\mu)}})\cdot 1_S \in H^*(\Hilb_n(S),\BC),\]
such that the product follows the standard order of the parts of $\mu$,\footnote{It is necessary put the order on the parts because the operators  $a_{k}(\gamma)$ for odd $\gamma$ and $k>0$ do not commute (they super-commute); for surfaces with the vanishing first Betti number $b_1(S)=0$ this is not relevant.}  
\[
1_S \in H^0(\Hilb_0(S),\BC)
\] is the identity class of a point, and 
\[
\mathfrak{z}(\mu):= |\Aut(\mu)| \prod_j \mu_j. 
\]
Note that $\Aut(\mu)$ denotes the automorphism group of the cohomologically weighted partition $\mu$, not of $(\mu_1, \dots, \mu_{\ell(\mu)})$.

By the results of Nakajima and Grojnowski \cite{Nak,Gro}, such classes provide a basis of  $H^*(\Hilb_n(S),\BC)$. They admit an alternative description in terms of nested Hilbert schemes, which we recall in Section \ref{epsilonNak}; it will be used for the construction of Nakajima--Grojnowski classes on interpolating spaces. 

With respect to the standard cohomology pairing on $H^*(\Hilb_n(S),\BC)$, we have
\begin{equation} \label{Hpairing}
\langle \lambda_{\sH}(\mu), \lambda_{\sH}(\eta) \rangle^{\sH_n}=(-1)^{\mathrm{age}(\mu)}\frac{\delta_{\mu,\eta^\vee}}{\mathfrak{z}(\mu)},
\end{equation}
where 
 \[
 \mathrm{age}(\mu)= n-\ell(\mu).
 \]
 The dual partition $\eta^\vee$ is a cohomologically weighted partitions with the same parts $\eta_j$ as $\eta$ but whose classes are dual with respect to the standard cohomological pairing on $ H^*(S, \BC)$,
\[
\eta^\vee=((\eta_1, \gamma_{i_1}^\vee),\dots, (\eta_{\ell(\eta)}, \gamma_{i_{\ell(\eta)}}^\vee)).
\]
   The identity (\ref{Hpairing}) follows from the commutation relations of the Heisenberg algebra acting on the cohomology of $\Hilb_n(S)$ via the Nakajima--Grojnowski operators. 
\subsection{Orbifold cohomology} Let 
\[
\Sym_n(S):=[S^n/\Sigma_n]
\]
 denote the orbifold symmetric product of $S$, where $\Sigma_n$ is the symmetric group on $n$ elements.  The orbifold cohomology,  introduced\footnote{Orbifold cohomology was also conceived by Kontsevich, but his results were never published; see  Appendix of \cite{Abra} for a historical account.} by Chen and Ruan \cite{ChenRuan},
 \[
 H_{\mathrm{orb}}^*(\Sym_n(S),\BC),
 \]
  is defined via the singular cohomology of the inertia stack $\CI \Sym_n(S)$, which is given by the following fiber product
  
  	\begin{equation*}
  	\begin{tikzcd}[row sep = scriptsize, column sep = scriptsize]
  		& \CI \Sym_n(S)  \arrow[d]  \arrow[r] &  \Sym_n(S)  \arrow[d,"\Delta"] \\
  		&  \Sym_n(S) \arrow[r,"\Delta"]&  \Sym_n(S)\times \Sym_n(S) ,
  	\end{tikzcd}
  \end{equation*}
 where 
\[ 
\Delta \colon \Sym_n(S) \rightarrow \Sym_n(S)\times  \Sym_n(S)
\]
is the diagonal embedding. 

Recall that  the inertia stack of a symmetric product can expressed as follows, 
\begin{equation} \label{Inertia}
\CI \Sym_n(S)=\coprod_{[\sigma]}[\Sym_n(S)^\sigma/C(\sigma)],
\end{equation}
where the disjoint union is taken over conjugacy classes $[\sigma]$  of the symmetric group $\Sigma_n$, the space $\Sym_n(S)^\sigma$ is the locus in $S^n$ fixed by a representative $\sigma$ of the conjugacy class $[\sigma]$, and $C(\sigma)$ is the centraliser of $\sigma$.

Conjugacy classes of the symmetric group are in one-to-one correspondence with partitions of $n$ by associating the cycle lengths of a  disjoint-cycle decomposition of a permutation. With respect to this correspondence, we have the following  identifications 
\begin{align}\label{ident1}
	\begin{split}
1\rightarrow \prod_s \BZ^{m_s}_{s} \rightarrow C(\sigma) & \rightarrow \prod_{s} \Sigma_{m_s} \rightarrow 1,  \\
 \Sym_n(S)^\sigma &\cong S^{\ell(\sigma)},
\end{split}
\end{align}
where 
\[
(1^{m_1},\dots, s^{m_s} )=(\mu_1, \dots, \mu_{\ell(\mu)})
\] is the partition associated to the conjugacy class of $\sigma$ by decomposing $\sigma$ into disjoint cyclic permutations and recording their lengths, and $\ell(\sigma)=\ell(\mu)$. The integers $m_s$ denote the multiplicities of the parts of the partition.  The subgroup  $\prod_s \BZ^{m_s}_{s}$  acts trivially on $S^{\ell(\sigma)}$, while 
\[
\prod_{s} \Sigma_{m_s}\cong \Aut(\mu_1,\dots,\mu_{\ell(\mu)}) 
\] permutes the factors of $S^{\ell(\sigma)}$ corresponding to equal parts. Consequently, observe the following identity, 
\begin{equation} \label{Csigma}
|C(\sigma)|=|\Aut(\mu_1,\dots,\mu_{\ell(\mu)})| \prod_j \mu_j. 
\end{equation}
With respect to the correspondence between the conjugacy classes of symmetric groups and partitions, we define the orbifold cohomology of a symmetric product, 
\[ 
H_{\mathrm{orb}}^*(\Sym_n(S),\BC):=H^{*-2 \mathrm{age}(\sigma)}(\CI\Sym_n(S),\BC),
\]
such that 
\[
 \mathrm{age}(\sigma)= \mathrm{age}(\mu)=n-\ell(\mu).
\]

Consider now a cohomologically weighted partition $\mu$ as in (\ref{part}). By using the identification (\ref{ident1}), we can associate to it a class in  $H_{\mathrm{orb}}^*(\Sym_n(S),\BC)$,
\vspace{0.2cm}
\[
\lambda_{\sS}(\mu)=\frac{1}{\mathfrak{z}(\mu)} \sum_{\rho \in C(\sigma)}\rho^*(\gamma_{i_1} \otimes\hdots \otimes \gamma_{i_{\ell(\mu)}} ) \in H^{*-2\mathrm{age}(\sigma)}(S^{\ell(\sigma)},\BC)^{C(\sigma)},
\] 
where we use the natural isomorphism between the cohomology of a finite-group quotient and the invariant part of the cohomology.  The classes $\lambda_S(\mu)$ span the $C(\sigma)$-invariant part of $H^{*}(S^{\ell(\sigma)},\BC)$, hence, they form a basis of $H_{\mathrm{orb}}^*(\Sym_n(S),\BC)$. See \cite{FB} for a detailed study of orbifold cohomologies of global quotients. 

A non-degenerate pairing on $H_{\mathrm{orb}}^*(\Sym_n(S),\BC)$ is defined to be the standard cohomological pairing on each component $[\Sym_n(S)^\sigma/C(\sigma)]$ of the inertia stack.  Using (\ref{Csigma}), a direct calculation yields 
\begin{equation} \label{Spairing}
\langle  \lambda_{\sS}(\mu), \lambda_{\sS}(\eta) \rangle^{\sS_n}=\frac{\delta_{\mu,\eta^\vee}}{\mathfrak{z}(\mu)},
\end{equation}
where the notation is the same as in (\ref{Hpairing}).  

\subsection{Isomorphism of cohomology groups with pairings} The cohomologies of both spaces,  
\[
H^*(\Hilb_n(S),\BC) \quad \text{and} \quad H_{\mathrm{orb}}^*(\Sym_n(S),\BC),
\] 
have a basis given by the cohomologically weighted partitions $\mu$. Moreover, their parings, (\ref{Hpairing}) and (\ref{Spairing}), agree up to a sign. In particular, after a choice of ordered basis of $H^*(S,\BC)$, we have an isomorphism of two vector spaces which respects the pairings, 
\begin{align} \label{isoL}
	\begin{split}
 H^*(\Hilb_n(S),\BC) &\xrightarrow{\sim}H_{\mathrm{orb}}^*(\Sym_n(S),\BC)\\
\lambda_\sH(\mu) &\mapsto (-i)^{\mathrm{age}(\mu)}\lambda_\sS(\mu),
\end{split}
\end{align}
where $i=\sqrt{-1}$. Note that the matching of intersection pairings fixes the above isomorphism only up to a sign, i.e., the same  isomorphism  but with the prefactor $i^{\mathrm{age}(\mu)}$ would also respect the intersection pairings. The sign is fixed by further matching the three-point functions provided by the Gromov--Witten invariants.   

For surfaces with the non-vanishing first Betti number $b_1(S)\neq0$, the isomorphism depends on the order of the basis  $\{\gamma_1, \dots, \gamma_r \}$; changing the order or the elements of the basis can change the sign of $L$ at some classes $\lambda_\sH(\mu) $. For this reason, the ordering is necessary. 
\subsection{Orbifold  ring structure}  The simplest way to define an associative ring structure on the orbifold cohomology  $H_{\mathrm{orb}}^*(\Sym_n(S),\BC)$ is to use the moduli space of degree $0$ representable maps from  $\p^1$ twisted at $(0,1, \infty)$, 
\[
\Mbar_{0,3}(\Sym_n(S), 0),
\]
which is constructed in \cite{AGV1, AGV}. 
More precisely, at $(0,1, \infty)$, the projective line $\p^1$ is an orbifold with $\BZ_k$-stabilisers for  varying $k \in \BZ_{>0}$. Since $\CI \Sym_n(S)$ can be viewed as the moduli space of maps from $B\BZ_k$ to $\Sym_n(S)$ for varying $k$, we have evaluation morphisms at $(0,1, \infty)$,
\[
\ev_j \colon \Mbar_{0,3}(\Sym_n(S), 0) \rightarrow \CI \Sym_n(S), \quad j=1, 2, 3. 
\]
The $3$-point Gromov--Witten invariants associated to  $\lambda_j \in H_{\mathrm{orb}}^*(\Sym_n(S),\BC)$ are defined by the integration over $\Mbar_{0,3}(\Sym_n(S), 0)$, 
\[ 
\langle \lambda_1, \lambda_2, \lambda_3\rangle^{\sS_n}_{0} =\int_{[\Mbar_{0,3}(\Sym_n(S), 0)]^\mathrm{vir}} \prod^3_{j=1}\ev_j^*\lambda_j. 
\]
Using the  non-degenerate intersection pairing and the $3$-point invariants, we define a product of two elements,
\[
\lambda_1\cdot \lambda_2 \in  H_{\mathrm{orb}}^*(\Sym_n(S),\BC)
\] 
by 
\[ 
\langle \lambda_1 \cdot \lambda_2, \lambda_3\rangle^{\sS_n}= \langle \lambda_1, \lambda_2, \lambda_3\rangle^{\sS_n}_0. 
\]
It is associative by the Witten--Dijkgraaf--Verlinde--Verlinde (WDVV) equation.  

\begin{remark}
In	\cite{AGV1, AGV}, two versions of moduli spaces of twisted maps were introduced. The first, denoted by $\CK_{g,k}(X,\beta)$, does not require the twisting at marked points to trivialise in families. The second, denoted by $\Mbar_{g,k}(X,\beta)$ in \cite[Section 6]{AGV}, is a gerbe of $\CK_{g,k}(X,\beta)$, over which the twisting becomes trivial. Only the latter admits evaluation morphisms to the inertia stack $\CI X$; for the former, one needs to consider a rigidified inertia stack. 
\end{remark}

The product above  admits a more explicit expression by unpacking the definition of $\Mbar_{0,3}(\Sym_n(S), 0)$. Indeed, this moduli space can be viewed as a space of degree $n$ covers of $\p^1$ ramified at $(0,1,\infty)$ together with a constant map from the source of a cover to $S$. With respect to this identification, $[\Mbar_{0,3}(\Sym_n(S), 0)]^\mathrm{vir}$ can be expressed in terms of the Chern classes of the Hodge bundle on moduli spaces of ramified covers. 

More generally, we can  define higher-genus degree 0 Gromov--Witten invariants of $\Sym_n(S)$ with any number of insertions, 
\[ 
\langle \lambda_1, \dots, \lambda_k\rangle^{\sS_n}_{g} =\int_{[\Mbar_{g,k}(\Sym_n(S), 0)]^\mathrm{vir}} \prod^k_{j=1}\ev_j^*\lambda_j,
\]
which admits a similar interpretation. 
 
\subsection{Quantum cohomology}

Let 
\[
E \in H_2(\Hilb_n(S),\BZ)
\]
 be the primitive exceptional curve class  of the Hilbert--Chow morphism 
\[
\Hilb_n(S) \rightarrow S^n/\Sigma_n.
\] 
In terms of Nakajima--Grojnowski classes, it is given by 
\begin{equation*}
	E=\mathfrak{z}(\mu) \lambda_{\sH}(\mu) \quad \text{for} \quad \mu= ((2,[\mathrm{pt}]),(1,[\mathrm{pt}]) \dots, (1,[\mathrm{pt}])).
\end{equation*}

Consider the moduli spaces of stable maps to $\Hilb_n(S)$ in an multiple of the primitive exceptional class $dE$, 
 \[ 
 \Mbar_{g,k}(\Hilb_n(S), d). 
 \]
 We define Gromov--Witten invariants of Hilbert schemes of points, 
\[ 
\langle \lambda_1, \dots, \lambda_k \rangle^{\sH_n}_{g,d} = \int_{[\Mbar_{g,k}(\Hilb_n(S), d)]^\mathrm{vir}} \prod^k_{j=1} \ev_j^* \lambda_j. 
\]
Genus $0$ invariants provide a one-parameter quantum deformation of the standard ring structure of $H^*(\Hilb_n(S),\BC)$, 
\[ 
\langle  \lambda_1 \cdot_q  \lambda_2, \lambda_3 \rangle^{\sH_n}= \sum_{d\geq 0} \langle \lambda_1, \lambda_2, \lambda_3  \rangle^{\sH_n}_{0,d}q^d,
\]
where $q$ is a formal variable. In particular,  this equips $H^*(\Hilb_n(S), \BC)[\! [ q]\!]$ with an associative ring structure. 
\subsection{Fulton--MacPherson compactification}\label{FMspaces} Let $\FM_{[n]}(S)$ be the ordered Fulton--MacPherson space of $n$ points, which is constructed as an iterated blow-up of $S^n$ in \cite{FM}, 
\[ 
\pi \colon \FM_{[n]}(S)= \mathrm{Bl}_{\tilde{\Delta}_2} \hdots \mathrm{Bl}_{\Delta_n}(S^n), 
\]
such that 
\[
\Delta_k:=\{(x_1,\dots,x_n) \in S^n \mid x_{i_{1}}=\hdots = x_{i_{k}}\} \subset S^n
\]
is the locus where at least $k$ points coincide. More precisely, it is a union of smooth irreducible components 
\[
\Delta_k=\bigcup_{|I|=k}\Delta_I,
\]
such that $\Delta_I$ is the locus where  points labelled by the subset $I \subseteq \{1, \hdots, n\}$ coincide. The preceding blow-ups separate  irreducible components whose indexing sets intersect,  $I\cap J\neq \emptyset$; while if $I\cap J= \emptyset$, then $\tilde{\Delta}_I$ and  $\tilde{\Delta}_J$ intersect transversely, making the order of blow-ups at $\tilde{\Delta}_I$ and  $\tilde{\Delta}_J$ irrelevant. By the blow-up at $\tilde{\Delta}_k$, we therefore mean blow-ups at the irreducible components $\tilde{\Delta}_I$ with respect to any order on $I$. 

By 
\[
\FM_{n}(S):=[\FM_{[n]}(S)/\Sigma_n]
\]
we will denote the unordered orbifold Fulton--MacPherson space.

The cohomology of Fulton--MacPherson spaces is of different nature. While for Hilbert schemes of points and symmetric products it is governed by partitions, for Fulton--MacPherson spaces the cohomology classes can be  labelled by rooted trees.  More precisely, we have 
\[ 
H^*(\FM_{[n]}(S),\BC)=H^*(S^n,\BC)[D_I]/\mathrm{relations},
\]
such that $D_I$ are boundary divisors, i.e., exceptional divisors associated to the blow-ups at $\tilde{\Delta}_I$. The above relations involve the Chern classes of $S$, diagonals of $S^n$  and $D_I$, as is shown in \cite{FM}. To a stable tree with $n$ leaves, we can associate a locus in $\FM_{[n]}(S)$, the generic point of which is  a marked degeneration of $S$ whose intersection graph is the given tree. Such loci can be realised as intersections of $D_I$. Heuristically, $D_I$ and their intersections are Fulton--MacPherson versions of classes associated to partitions $\mu$ discussed in previous sections.

We will not need the exact presentation of the cohomology, as it is not particularly useful for the intersection theory. Our primary interest lies in 
 the $\psi$-classes, 
\[ 
\Psi_i=\sum_{i\in I}D_I, \quad i=1,\dots,n. 
\]
 By \cite{NHilb}, they also admit a different interpretion in terms of rank 2 vector bundles $\BL_i$ associated to the cotangent spaces at the marked points, 
\[
\Psi_i=\frac{1}{2}(\mathrm{c}_1(\BL_i)+\pi_i^*\mathrm{c}_1(S)),
\]
where $\pi_i$ is given by composing $\pi$ with the projection to the $i$-th factor $p_i \colon S^n \rightarrow S$, 
\[ 
\pi_i:= p_i \circ \pi \colon  \FM_{[n]}(S) \rightarrow S. 
\] 

For Theorems \ref{wsH} and \ref{wsS}, we will need the moduli space of degree 0 maps to $\FM_{[n]}(S)$, 
\[ 
\Mbar_{g,k}\times \FM_{[n]}(S),
\]
together with the evaluation morphisms 
\[
\pi_i \colon \FM_{[n]}(S) \times \Mbar_{g,k}  \rightarrow S \times \Mbar_{g,k},
\]
which are defined by taking the projections on the first factor and the identity on the second. 
Using these structures, we can define the following integrals valued in the cohomology of $\Mbar_{g,k}$, 
\[ 
\langle  \Psi^{h_1}_1 \gamma_1 , \hdots , \Psi^{h_n}_n \gamma_n\rangle^{\FM_n}_{g,k} = \int_{[\FM_{[n]}(S)\times \Mbar_{g,k }] } \prod^{n}_{i=1} \Psi_i^{h_i}\pi_i^* \gamma_i 
\]
where $\gamma_i \in  H^*(S\times \Mbar_{g,k}, \BC)$ are some cohomology classes. Note that we integrate over the standard fundamental class, not the virtual one. 
Dilaton, string, and divisor equations (see \cite{NHilb} for the first two) give an effective method to deal with the integrals of $\Psi$-classes.
\section{Interpolating spaces}
\subsection{Moduli spaces of $\epsilon$-weighted subschemes}
We recall the definition of moduli spaces of $\epsilon$-weighted subschemes $\Hilb^{\epsilon}_n(S)$,  which interpolate between space $\Hilb_n(S)$ and $\FM_n(S)$. For the wall-crossing formula, we will need a slightly more general framework which allows  extra $m$ ordered markings on Fulton--MacPherson (FM) degenerations, $\Hilb^{\epsilon}_{n,[m]}(S)$, which were introduced in \cite[Section 2.5]{NHilb}. In what follows, $\ell_x(Z)$ denotes the length of a zero-dimensional subscheme at the point $x$, and $\ell(Z)$ denotes its total length.

 	\begin{definition} \label{defnweight}
	Let $\epsilon \in \BR_{>0}$ be a real number.  Consider a FM degeneration $W$  together with a zero-dimensional subscheme $Z\subset W$ and $m$ distinct ordered point $\underline{x}:=\{x_1,\dots,x_m\} \subset W$. The triple $(W,Z,\underline{x})$ is $\epsilon$-\textit{weighted} if 
	\begin{enumerate} 
		\item $\underline{x} \cap Z=\emptyset$,
		\item $\underline{x} \subset W^{\mathrm{sm}}$ and $Z\subset W^{\mathrm{sm}}$, 
		\item for all $x \in W$, we have $\ell_x(Z)\leq 1/\epsilon$,
		\item  for all end components\footnote{An end component of a FM degeneration corresponds to a vertex of valency 1 in the corresponding intersection graph, except if it is the vertex associated to $S$.  }  $\p^{2} \subset W$, such that $\underline{p} \cap \p^{2}=\emptyset$,  we have $\ell(Z|_{\p^{2}})>1/\epsilon$, 
		\item the group $\Aut(W, Z,\underline{x})$ is finite.
	\end{enumerate} 	
\end{definition}

A moduli space of $\epsilon$-weighted subschemes is defined as the open substack of the relative Hilbert schemes, 
\[
\Hilb^{\epsilon}_{n,[m]}(S) \subseteq \Hilb_n(\CW/\mathcal{FM}_{[m]}(S)),
\]
where 
\[
\CW \rightarrow \mathcal{FM}_{[m]}(S)
\] is the universal FM degeneration over the moduli stack of not necessarily stable  FM degenerations with $m$ marked points. The space $\Hilb^{\epsilon}_{n,[m]}(S)$ is a proper and smooth Deligne--Mumford stack by the results from \cite[Section 5]{NHilb}.

\subsection{Nakajima--Grojnowski classes} \label{epsilonNak} 
 Although the Nakajima--Grojnowski correspondences can  be similarly  constructed for $\Hilb^{\epsilon}_{n,[m]}(S)$, the nested Hilbert schemes approach is more suitable for the definition of Nakajima--Grojnowski classes on  $\Hilb^{\epsilon}_{n,[m]}(S)$. 

For a cohomologically weighted partition $\mu$, as in (\ref{part}),   let 
\[
\Hilb^{\epsilon}_{\mu,[m]}(S)
\]
be the $\mathcal{FM}_{[m]}(S)$-relative moduli spaces of flags of $0$-dimensional subschemes. It parametrises triples 
\[
(W,\underline{Z},\underline{x}),
\] 
where 
\[
\underline{Z}=\{   Z_1 \subset Z_2  \subset \hdots \subset Z_{\ell(\mu)}\}
\]
is a flag, such that $(W,Z_{\ell(\mu)},\underline{p})$ is $\epsilon$-weighted, and 
\[ 
\CO_{Z_{i}}/\CO_{Z_{i-1}}
\]
is a length $\mu_i$ sheaf supported on a single point; the flags are required to respect the standard order of the partition $\mu$. There is a natural projection to $S^{\ell(\mu)}$ given by sending a flag $\underline{Z}$ to the supports of the quotients followed by the contraction of FM degenerations, 

\begin{align} \label{correspondence0}
\begin{split}
q \colon \Hilb^{\epsilon}_{\mu,[m]}(S)& \rightarrow S^{\ell(\mu)}, \\
\underline{Z} &\mapsto (\mathrm{Supp}(\CO_{Z_{1}}),\hdots, \mathrm{Supp}(\CO_{Z_{\ell(\mu)}}/\CO_{Z_{\ell(\mu)-1}}) ).
\end{split}
\end{align}

There is also a projection given by forgetting all but the last element of the flag,
\begin{align*}
p \colon \Hilb^{\epsilon}_{\mu,[m]}(S) &\rightarrow \Hilb^{\epsilon}_{n,[m]}(S), \\
\underline{Z} &\mapsto Z_{\ell(\mu)}.
\end{align*}
The properness of the morphism $p$ readily follows from the properness of relative nested Quot schemes; we leave the details to the reader. We define a Nakajima--Grojnowski class associated to the cohomologically weighted partition $\mu$, 
\begin{align*}
\lambda_{\sH}(\mu):= \frac{1}{\mathfrak{z}(\mu)}  p_*(q^*(\gamma_{i_1}\otimes \hdots \otimes \gamma_{i_\ell}) \cap  [\Hilb^{\epsilon}_{\mu,[m]}(S)]) \in H^*(\Hilb^{\epsilon}_{n,[m]}(S)).
\end{align*}

For Hilbert schemes $\Hilb_n(S)$, this agrees with the classes constructed by the Nakajima--Grojnowski correspondences; see, for example, the end of the proof of \cite[Theorem 2.4]{NOY}. The (complex) codimension of the classes  $\lambda_{\sH}(\mu)$ are equal to 
\[ 
n-\ell(\mu)+\sum \deg_\BC(\gamma_{i_1}).
\]
Hence, their dimension increases as the number of marked points $m$ goes up. 

\begin{remark} \label{classes} For effective algebraic classes $\gamma_{i_j}$, the classes $\lambda_{\sH}(\mu)$ are  closures of the standard Nakajima--Grojnowski classes supported on the open dense locus of $\Hilb_{n,[m]}^{\epsilon}(S)$ parametrizing $\epsilon$-weighted  subschemes on $S$ (i.e., the FM degeneration $W$ is $S$ itself). 
\end{remark}

\subsection{Exceptional curve class} \label{exc} If $m=0$, for all $\epsilon<1$, the spaces $\Hilb_{n,[m]}^{\epsilon}(S)$ possess an \textit{exceptional curve class} $E$, which is defined as 
\[ 
E= \mathfrak{z}(\mu) \lambda_{\sH}(\mu)  \quad \text{for} \quad \mu= ((2,[\mathrm{pt}]),(1,[\mathrm{pt}]) \dots, (1, [\mathrm{pt}])). 
\]
Geometrically, this class is represented by a curve which is given by letting the length-2 reduced structure vary on one point in $S$, while the other $n-1$ points are distinct and fixed. Note that it is a zero class on $\FM_{n}(S)$, i.e., when $\epsilon=1$. This will make the Gromov--Witten theory of  $\FM_{n}(S)$ mostly  irrelevant for the purposes of the present work.

If $m>0$, then the similarly defined class  $E$ is only a curve class relative to $\mathcal{FM}_{[m]}(S)$, i.e., after fixing the points indexed by $[m]$. It admits a similar  description: it is the class of the closure of the locus given by letting the length-2 reduced structure vary on one point in $S$, the other $n-1$ points are distinct and fixed, and the points labelled by $[m]$ roam around freely on $S$. For $m>0$, the above discussion motivates the definition of the following curve class, 
\begin{equation} \label{topterm}
E_m:=E \cap \prod^m_{i=1} \pi_i^*[\mathrm{pt}],
\end{equation}
such that $\pi_i \colon \FM_{n,[m]}(S) \rightarrow S$ are  as in Section \ref{FMspaces}. 
\subsection{Moduli spaces of $\epsilon$-weighted points}
We recall the definition of moduli spaces of $\epsilon$-weighted points $\Sym^{\epsilon}_n(S)$,  which interpolate between spaces $\Sym_n(S)$ and $\FM_n(S)$. They were defined by Hassett \cite{Hass} in dimension 1 and generalised by Routis \cite{Rou} to arbitrary dimensions; we also refer to \cite{NHilb}. As for Hilbert schemes, we will need a more general interpolating space which allows extra $m$  ordered markings on FM degenerations. 

\begin{definition}
Let  $\epsilon \in \BR_{>0}$ be a real number.  Let $W$ be a FM degeneration together with two sets of ordered points, $\underline{x}'=\{x_1,\dots,x_n\} \subset W$ and $\underline{x}=\{x_{n+1},\dots,x_{m+n}\} \subset W$, such that the points in $\underline{x}$ are pairwise distinct.  The triple $(W, \underline{x}', \underline{x})$ is $\epsilon$-\textit{weighted} if 
\begin{enumerate} 
	\item  $\underline{x}' \cap  \underline{x}=\emptyset$,
	\item $\underline{x}' \subset W^{\mathrm{sm}}$ and $\underline{x} \subset W^{\mathrm{sm}}$, 
	\item for all $x \in W$, we have $\mathrm{mult}_x(\sum^{n}_{i=1} x_i)\leq 1/\epsilon$, 
	\item  for all end components $\p^{2} \subset W$, such that $\underline{x} \cap \p^2=\emptyset$, we have 
	 $\mathrm{mult}({\sum^{n}_{i=1} x_i}|_{\p^{2}})>1/\epsilon$, 
	\item the group $ \Aut(W, \underline{x}', \underline{x})$ is finite.
\end{enumerate} 	
\end{definition}

We define a moduli space of ordered $\epsilon$-weighted points as the open substack of the relative $n$-fold products, 
\[
\FM^{\epsilon}_{[n+m]}(S) \subseteq \CW^n:= \CW\times_{\mathcal{FM}_{[m]}(S)}\hdots \times_{\mathcal{FM}_{[m]}(S)}\CW,
\]
where, as before, $\CW \rightarrow \mathcal{FM}_{[m]}(S)$ is the universal FM degeneration over the moduli stack of not necessarily stable  FM degenerations with $m$ marked points  $\mathcal{FM}_{[m]}(S)$. They are smooth and proper varieties.  These spaces also admit blow-up description by \cite{Rou}. Such a blow-up description is currently missing for $\Hilb^{\epsilon}_{n,[m]}(S)$.  

The unordered version of these spaces is given by the following groupoid construction 
\[
\Sym^{\epsilon}_{n,[m]}(S):= [\FM^{\epsilon}_{[n+m]}(S)/\Sigma_n],
\]
where $\Sigma_n$ acts on the first $n$ points. The resulting space retains the projection morphism
\[ 
p \colon \Sym^{\epsilon}_{n,[m]}(S) \rightarrow \mathcal{FM}_{[m]}(S),
\]
which can be most directly seen by interpreting $\Sym^{\epsilon}_{n,[m]}(S)$ as a space of (partially) unordered sections on FM degenerations.

\subsection{Orbifold classes} As for moduli spaces of $\epsilon$-weighted subschemes, we will construct orbifold classes on $\Sym^{\epsilon}_{n,[m]}(S)$ using the $\mathcal{FM}_{[m]}(S)$-relative versions of the inertia stack.  Consider the $\mathcal{FM}_{[m]}(S)$-relative inertia stack of $\Sym^{\epsilon}_{n,[m]}(S)$, 

	\begin{equation*}
	\begin{tikzcd}[row sep = scriptsize, column sep = scriptsize]
		&\CI\Sym^{\epsilon}_{n,[m]}(S)  \arrow[d]  \arrow[r] &  \Sym^{\epsilon}_{n,[m]}(S)  \arrow[d,"\Delta"] \\
		& \Sym^{\epsilon}_{n,[m]}(S) \arrow[r,"\Delta"]&   \Sym^{\epsilon}_{n,[m]}(S)\times_{\mathcal{FM}_{[m]}(S)} \Sym^{\epsilon}_{n,[m]}(S),
	\end{tikzcd}
\end{equation*}
where 
\[ 
\Delta\colon  \Sym^{\epsilon}_{n,[m]}(S) \rightarrow \Sym^{\epsilon}_{n,[m]}(S)\times_{\mathcal{FM}_{[m]}(S)} \Sym^{\epsilon}_{n,[m]}(S)
\]
is the $\mathcal{FM}_{[m]}(S)$-relative diagonal. The relative inertia stack parametrises pairs 
\[ 
(A,\alpha), 
\] 
where $A$ is an object in the stack $\Sym^{\epsilon}_{n,[m]}(S)$ (viewed as a fibered category), and $\alpha$ is an automorphism of $A$, such that $p(\alpha)=\mathrm{id}$ for the  projection morphism   
\[
p \colon \Sym^{\epsilon}_{n,[m]}(S) \rightarrow \mathcal{FM}_{[m]}(S),
\] 
as explained, for example, in \cite[Tag 050P]{stacks-project}. Since 
\[
\FM^{\epsilon}_{[n+m]}(S) \rightarrow  \mathcal{FM}_{[m]}(S)
\]  is representable  by schemes, such automorphisms of objects arise  through the action of $\Sigma_n$ on configurations  with repeating points. Hence, we obtain a description similar to (\ref{Inertia}),

\[
\CI\Sym^{\epsilon}_{n,[m]}(S)= \coprod_{[\sigma]} [\Sym^{\epsilon}_{n,[m]}(S)^\sigma/C(\sigma)].
\]
where, for an element $\sigma \in \Sigma_n$, 
\begin{equation} \label{fixedlocus}
\Sym^{\epsilon}_{n,[m]}(S)^\sigma:= \{x \in \FM^{\epsilon}_{[n+m]}(S)  \mid \sigma \cdot x\cong x, \quad p(\sigma \cdot x\cong x)=\id \},
\end{equation}
i.e., it is the locus fixed by the group element $\sigma$ relatively to $\mathcal{FM}_{[m]}(S)$. In particular, it excludes points fixed by $\sigma$ only up to automorphisms of FM degenerations.

Moreover, there exists a projection morphism given by contracting bubbles of FM degenerations 
\[ 
\tau \colon \Sym^{\epsilon}_{n,[m]}(S)^\sigma \rightarrow S^{\ell(\sigma)},
\]
which preserves the action of the centraliser $C(\sigma)$. 
We then define the orbifold classes associated to the cohomologically weighted partitions as before, 
\begin{multline*}
\lambda_{\sS}(\mu):= \frac{1}{\mathfrak{z}(\mu)}  \tau^* \bigg(\sum_{\rho \in C(\sigma)}\rho^*(\gamma_{i_1} \otimes\hdots \otimes \gamma_{i_{\ell(\mu)}} ) \bigg) \\
\in H^{*-2\mathrm{age}(\sigma)}([\Sym^{\epsilon}_{n,[m]}(S)^\sigma/C(\sigma)],\BC).
\end{multline*}
Remark \ref{classes}  also applies in this case: for effective classes $\gamma_{i_j}$, the  classes $\lambda_\sS(\mu)$ are closures of the standard orbifold classes supported on the open dense subset of $\Sym^{\epsilon}_{n,[m]}(S)$ parametrising $\epsilon$-weighted points on $S$.  
\section{Maps to interpolating spaces}
\subsection{Overview}
Both interpolating spaces $\Hilb^{\epsilon}_{n,[m]}(S)$ and  $\Sym^{\epsilon}_{n,[m]}(S)$ are smooth and proper. Hence,  the moduli spaces of stable maps to these spaces are also proper and carry perfect obstruction theories. 

First, for the exceptional curve class $E_m$ defined in Section \ref{exc}, we have a moduli space of stable maps in the class $dE_m$, 
\[ 
\Mbar_{g,k}(\Hilb^{\epsilon}_{n,[m]}(S), d). 
\]
Note that the interpolating spaces $\Hilb^{\epsilon}_{n,[m]}(S)$ are Deligne--Mumford stacks in general, hence it is necessary to consider twisted maps of \cite{AGV1, AGV}.

As we show  in Proposition \ref{Proprel}, there also exists a  proper $\mathcal{FM}_{[m]}(S)$-relative moduli space of maps, denoted by $\Mbar_{g,k}^{\mathsf{rel}}(\Hilb^{\epsilon}_{n,[m]}(S), d)(\BC)$, which is  open in $\Mbar_{g,k}(\Hilb^{\epsilon}_{n,[m]}(S), d)$ by Remark \ref{open}. The main feature of this relative space is that we can apply various operations to $\mathcal{FM}_{[m]}(S)$ (like blow-ups) without affecting the relative virtual fundamental classes. Also, in the case of  $\Hilb^{\epsilon}_{n,[m]}(S)$, no twisted maps are necessary because it is representable by schemes relatively to $\mathcal{FM}_{[m]}(S)$. 
The same applies $\Sym^{\epsilon}_{n,[m]}(S)$, except the last point. The spaces $\Mbar_{g,k}^{\mathsf{rel}}(\Hilb^{\epsilon}_n(S), d)$ will be used to relate the Gromov--Witten theories of $\Hilb_{n}(S)$ and $\Sym_{n}(S)$

\subsection{Relative  moduli spaces of maps}\label{Relative} 
By the construction of $\Hilb^{\epsilon}_{n,[m]}(S)$ as a $\mathcal{FM}_{[m]}(S)$-relative moduli space of subschemes, it admits a non-proper  morphism representable by schemes
\[ 
p\colon \Hilb^{\epsilon}_{n,[m]}(S) \rightarrow \mathcal{FM}_{[m]}(S). 
\]
Let 
\[ 
\Mbar_{g,k}^{\mathsf{rel}}(\Hilb^{\epsilon}_{n,[m]}(S), d) \rightarrow \mathcal{FM}_{[m]}(S) 
\]
be the  $\mathcal{FM}_{[m]}(S)$-relative moduli space of stable maps to $\Hilb^{\epsilon}_{n,[m]}(S)$ in the class $dE$. Its $B$-valued points are commutative diagrams

	\begin{equation} \label{factorisation}
	\begin{tikzcd}[row sep = scriptsize, column sep = scriptsize]
		& (\CC, \underline{p}) \arrow[d]  \arrow[r,"f"] &  \Hilb^{\epsilon}_{n,[m]}(S) \arrow[d,"p"] \\
		& B \arrow[r]&  \mathcal{FM}_{[m]}(S),
	\end{tikzcd}
\end{equation}
such that for all closed points of $B$, the class of the map is given by  
\[
f_*[C]=\iota^!(dE) \in H_2(\Hilb^{\epsilon}_{n,[m]}(S)\times_{\mathcal{FM}_{[m]}(S)} \Spec(\BC),\BQ),
\]
where $\iota \colon \Spec(\BC) \rightarrow  \mathcal{FM}_{[m]}(S)$ is the inclusion of a point; note that by definition we have  $E_m=\iota^!(dE)$ after pushing forward to $\Hilb^{\epsilon}_{n,[m]}(S)$. 
Since $p$ is not proper, there is no a priori reason for this space to be proper. In Proposition \ref{Proprel}, we will establish its properness. 

The space  
\[
\Mbar_{g,k}^{\mathsf{rel}}(\Sym^{\epsilon}_{n,[m]}(S), 0)
\] 
is defined similarly,  although, since $\Sym^{\epsilon}_{n,[m]}(S)\rightarrow  \mathcal{FM}_{[m]}(S)$ is only representable by Deligne--Mumford stacks, one has to use the twisted maps of \cite{AGV1, AGV}. Observe that for $\Hilb_n(S)$ and $\Sym_n(S)$, these relative moduli spaces of maps are the standard spaces discussed in Section \ref{SecPrel}.

\subsection{Properness} In this section, we establish the properness of relative moduli spaces of maps.  We start with a lemma that justifies the name ``exceptional curve class" for $E_m$ on $\Hilb^{\epsilon}_{n,[m]}(S)$. 

First, let 
\[
\FM^{\epsilon}_{[n+m]}(S)/\Sigma_n
\] 
be the $\mathcal{FM}_{[m]}(S)$-relative coarse quotient of the $\Sigma_n$-action, i.e., it rigidifies only those stabalisers that are trivial after projection to $\mathcal{FM}_{[m]}(S)$. More precisely, such a coarse quotient is given by passing to a smooth chart $U$ of $\mathcal{FM}_{[m]}(S)$ 
\[
\FM^{\epsilon}_{[n+m]}(S)|_U,
\] 
then taking the coarse quotient 
\[
\FM^{\epsilon}_{[n+m]}(S)|_U/\Sigma_n 
\]
and finally descending this back to $\mathcal{FM}_{[m]}(S)$; it is clear that on double intersections the  isomorphisms of the universal FM degeneration $\CW_{|U}$   induce the isomorphisms of $\FM^{\epsilon}_{[n+m]}(S)|_U/\Sigma_n $  which moreover satisfy the cocycle condition on the triple intersections. 

The resulting space retains the projection morphism  
\[
\FM^{\epsilon}_{[n+m]}(S)/\Sigma_n \rightarrow \mathcal{FM}_{[m]}(S).
\]
There also exists a Hilbert--Chow morphism, constructed as a $\mathcal{FM}_{[m]}(S)$-relative Hilbert--Chow morphism, 

\begin{align*}
	\pi \colon \Hilb^{\epsilon}_{n,[m]}(S) &\rightarrow  \FM^{\epsilon}_{[n+m]}(S)/\Sigma_n \\
	(W,Z,\underline{x}) &\mapsto(W,  \mathrm{Supp}(Z),\underline{x}),
\end{align*} 
where the support $\mathrm{Supp}(Z)$ is weighted by the lengths of subschemes.
\begin{lemma} \label{exceptionalcurve}  Let $\pi$ be the Hilbert--Chow morphism,  then we have 
	\[ 
	\pi_*(E_m)=0
	\]
	for the exceptional curve class $E_m$. 
\end{lemma}
\begin{proof}
	As explained in Section \ref{exc}, the class $E_m$ can be represented by a curve which is given by letting the length-2 reduced structure vary on one point in $S$, while the other $n+m-1$ points are distinct and fixed. Such a curve is clearly contracted to a point by $\pi$. Hence, we must have $\pi_*(E_m)=0$.     
\end{proof}

\begin{corollary} \label{threefold}
Consider a map 
\[
[f \colon (C,\underline{p}) \rightarrow \Hilb^{\epsilon}_{n,[m]}(S)] \in \Mbar_{g,k}^{\mathsf{rel}}(\Hilb^{\epsilon}_n(S), d)(\BC).
\] 
Let $  \CW \rightarrow C$ be the associated family of FM degenerations, and $\Gamma \subset \CW$ be the associated family of 0-dimensional subschemes. Then 
	\[
	\CW\cong W \times C
	\]
	for some FM degeneration $W$, and $\Gamma$ is supported on 
	\[
	\{ x_1, \hdots, x_\ell\}\times C \subset W \times C
	\]
	for some points $x_i \in W$. 
	
	The same holds for maps inside $\Mbar_{g,k}^{\mathsf{rel}}(\Sym^{\epsilon}_{n,[m]}(S), 0)$, i.e., the associated maps to $W\times C$ are ramified covers of $\{ x_1, \hdots, x_\ell\}\times C$. 
\end{corollary}	

\begin{proof}
The first claim follows from the definition of $\Mbar_{g,k}^{\mathsf{rel}}(\Hilb^{\epsilon}_n(S), d)$ - all closed points of the moduli space fit into the commutative diagrams, 
	\begin{equation*}
	\begin{tikzcd}[row sep = scriptsize, column sep = scriptsize]
		& (C, \underline{p}) \arrow[d]  \arrow[r,"f"] &  \Hilb^{\epsilon}_{n,[m]}(S) \arrow[d,"p"] \\
		& \Spec(\BC) \arrow[r]&  \mathcal{FM}_{[m]}(S),
	\end{tikzcd}
\end{equation*}
which implies that the family $\CW$ must be trivial. 

 For the second claim, we use Lemma \ref{exceptionalcurve}, which implies that  the map $\pi \circ f$ is constant. Indeed, if $\Gamma$ was not supported on $\{ x_1, \hdots, x_\ell\}\times C$, then the fibers of $\Gamma$ over $C$ has a non-constant support, contradicting Lemma \ref{exceptionalcurve}. 
 
For $\Mbar_{g,k}^{\mathsf{rel}}(\Sym^{\epsilon}_{n,[m]}(S), 0)$, the same analysis applies; since the curve class is trivial, the conclusion of Lemma \ref{exceptionalcurve} is automatically satisfied. 
\end{proof}	

For next proposition, Corollary \ref{threefold} is not needed, but it provides a geometric reason for the result to hold. 

\begin{proposition} \label{Proprel}
	The relative moduli spaces of maps
	\[ 
	\Mbar_{g,k}^{\mathsf{rel}}(\Hilb^{\epsilon}_{n,[m]}(S), d) \quad \text{and} \quad	\Mbar_{g,k}^{\mathsf{rel}}(\Sym^{\epsilon}_{n,[m]}(S), 0)
	\]
	are proper.  

\end{proposition}	
\begin{proof} 
	We prove the statement for 	$\Hilb^{\epsilon}_{n,[m]}(S)$. For  $\Sym^{\epsilon}_{n,[m]}(S)$, as the  maps are of degree $0$, the result follows immediately from the properness of  $\Sym^{\epsilon}_{n,[m]}(S)$.  
	
Since $\Hilb^{\epsilon}_{n,[m]}(S)$ is proper, it is enough to prove the existence part of the valuative criterion.  	Let $\Delta$ be the spectrum of a discrete valuation ring, and $\Delta^*$ be its generic generic point. Consider a relative family of maps over $\Delta^*$,
	
	\begin{equation*}
	\begin{tikzcd}[row sep = scriptsize, column sep = scriptsize]
		& \CC^* \arrow[d]  \arrow[r,"f^*"] &  \Hilb^{\epsilon}_{n,[m]}(S) \arrow[d] \\
		&  \Delta^*  \arrow[r] &  \mathcal{FM}_{[m]}(S).
	\end{tikzcd}
\end{equation*}
Possibly after a finite base change (which we suppress from the notation),  $\CC^*$ admits a section, 
	\[ 
	\Delta^* \rightarrow \CC^*.
	\]
	After composing with $f^*$, we obtain a family of $0$-dimensional subschemes
	\[ 
	\Delta^* \rightarrow \Hilb^{\epsilon}_{n,[m]}(S),
	\]
	which, possibly after another finite base change, admits an extension over $\Delta$ 	by the properness of  $\Hilb^{\epsilon}_{n,[m]}(S)$, 
	\[ 
	p \colon \Delta \rightarrow \Hilb^{\epsilon}_{n,[m]}(S).
	\]
The family $p$ in turn gives a family of FM degenerations
	\[ 
q \colon	\Delta \rightarrow \mathcal{FM}_{[m]}(S). 
	\]

Next, the space $\Hilb^{\epsilon}_{n,[m]}(S)$ is an open substack of the $\mathcal{FM}_{[m]}(S)$-relative Hilbert scheme of points (without any additional conditions),
\[
\Hilb^{\epsilon}_{n,[m]}(S) \subseteq \Hilb_{n}(\CW/ \mathcal{FM}_{[m]}(S)). 
 \]
The latter space is projective over $\mathcal{FM}_{[m]}(S)$ by the projectivity of Hilbert schemes of points. Hence, by the relative properness of stable maps to relatively projective targets, possibly after a finite base change, we obtain  the family of maps over $\Delta$, which commutes with the family of FM degenerations given by  $q$,  
	\begin{equation*}
	\begin{tikzcd}[row sep = scriptsize, column sep = scriptsize]
		& \CC \arrow[d]  \arrow[r,"f"] &  \Hilb_{n}(\CW/ \mathcal{FM}_{[m]}(S)) \arrow[d] \\
		&  \Delta  \arrow[r, "q"] &  \mathcal{FM}_{[m]}(S).
	\end{tikzcd}
\end{equation*}

We have to show that the image of $f$ is contained inside $\Hilb^{\epsilon}_{n,[m]}(S)$. First, possibly after another finite base change, there exists a section of $\CC$, 
\[
\Delta \rightarrow \CC.
\]
By composing it with $f$, we obtain 
\[ 
p' \colon \Delta \rightarrow \Hilb_{n}(\CW/ \mathcal{FM}_{[m]}(S)). 
\]
We claim that if $p'$ lies in $\Hilb^{\epsilon}_{n,[m]}(S)$, then so does $f$.\footnote{Observe that this claim is the key step in the proof, and it relies on the assumption that the curve class is exceptional; such a conclusion does not hold for other curve classes. } Indeed, consider the relative Hilbert--Chow morphism 
\[
\pi \colon  \Hilb_{n}(\CW/ \mathcal{FM}_{[m]}(S)) \rightarrow \CW^n/\Sigma_n,
\]
where $\CW^n$ is the $\mathcal{FM}_{[m]}(S)$-relative $n$-fold product of $\CW$, and the coarse quotient is defined as in the beginning of this section. Then, by Lemma \ref{exceptionalcurve}, the class of $f$ at the central fiber of $\CW^n/\Sigma_n$ is zero, since it is zero at the generic fiber by  assumption. We conclude that the curve $\CC$ is contracted to a point at the central fiber of  $\CW^n/\Sigma_n$. The $\epsilon$-weightedness involves only lengths of subschemes, hence,  $\Hilb^{\epsilon}_{n,[m]}(S)$ is given by the preimage of the $\epsilon$-weighted locus inside $\CW^n/\Sigma_n$ with respect to $\pi$
\[ 
\Hilb^{\epsilon}_{n,[m]}(S)=\pi^{-1}(\FM^{\epsilon}_{[n+m]}(S)/\Sigma_n).
\]
 This implies that the image of $f$ lies in $\Hilb^{\epsilon}_{n,[m]}(S)$ if and only if the corresponding point in $\CW^n/\Sigma_n$ lies in $\FM^{\epsilon}_{[n+m]}(S)/\Sigma_n$. Since $\CC$  is connected, it is therefore enough to check this at $p'$, which proves the claim. 

It therefore remains to establish that the family   $p'$  lies in $\Hilb^{\epsilon}_{n,[m]}(S)$. By construction, $p'$ commutes with $q$, and by the relative properness of $\Hilb_{n}(\CW/ \mathcal{FM}_{[m]}(S))$  it therefore must agree with $p$. This proves the statement. 
\end{proof}

From now on, we will be working with the relative spaces of maps. However, as explained in the following remark, they are just connected components of the standard moduli spaces of twisted maps to $\Hilb^{\epsilon}_{n,[m]}(S)$ and $\Sym^{\epsilon}_{n,[m]}(S)$ (this will be irrelevant to us).

\begin{remark} \label{open} There exist a natural morphism 
	\[
		 \Mbar_{g,k}^{\mathsf{rel}}(\Hilb^{\epsilon}_{n,[m]}(S), d) \rightarrow \Mbar_{g,k}(\Hilb^{\epsilon}_{n,[m]}(S), d)
		 \]
given by forgetting the map from the base $B \rightarrow 	\mathcal{FM}_{[m]}(S)$. In fact, it is an open embedding.  Indeed, by Lemma \ref{exceptionalcurve}, the existence of the commutative  diagram  (\ref{factorisation}) over $B=\Spec(\BC)$ for a map in $\Mbar_{g,k}(\Hilb^{\epsilon}_{n,[m]}(S), d)$ is equivalent to the condition that the stabilisers of the twisted points do not map to the stabilisers of the FM degenerations.  This is an open condition. Hence, 
\[
\Mbar_{g,k}^{\mathsf{rel}}(\Hilb^{\epsilon}_{n,[m]}(S), d) \subseteq \Mbar_{g,k}(\Hilb^{\epsilon}_{n,[m]}(S), d)
\]
is a  proper open substack.  For $\Hilb^{\epsilon}_{n,[m]}(S)$, this condition prohibits twisting completely.  For $\Sym^{\epsilon}_{n,[m]}(S)$, the same discussion applies with the difference that the stabalisers can only map into the stabalisers associated to the $\Sigma_n$-action. 
\end{remark}	

\subsection{Obstruction theory}
Let  $T^{\mathsf{rel}}$ be the relative tangent bundle associated to the morphism 
\[
\Hilb^{\epsilon}_{n,[m]}(S) \rightarrow \mathcal{FM}_{[m]}(S). 
\]
Consider the universal curve 
 \[
 \pi \colon \CC \rightarrow \Mbar^{\mathsf{rel}}_{g,k}(\Hilb^{\epsilon}_{n,[m]}(S), d)
 \]
 and universal map
 \[
 F \colon  \CC \rightarrow \Hilb^{\epsilon}_{n,[m]}(S). 
 \] 
By \cite{BF}, the complex 
\[
R\pi_*(F^*(T^{\mathsf{rel}}))
\] 
equips $\Mbar^{\mathsf{rel}}_{g,k}(\Hilb^{\epsilon}_{n,[m]}(S), d)$ with a   perfect obstruction theory   relative to  $\mathcal{FM}_{[m]}(S)$. Its relative virtual dimension is $0$. The dimension of $\mathcal{FM}_{[m]}(S)$ is $m$. Hence,  the absolute virtual dimension of  $\Mbar^{\mathsf{rel}}_{g,k}(\Hilb^{\epsilon}_{n,[m]}(S), d)$ is equal to $m$. The same applies to 
\[
\Mbar^{\mathsf{rel}}_{g,k}(\Sym^{\epsilon}_{n,[m]}(S), 0).
\]
Let
\[ 
[\Mbar_{g,k}^{\mathsf{rel}}(\Hilb^{\epsilon}_{n,[m]}(S),d)]^{\mathrm{vir}} \quad \text{ and } \quad [\Mbar_{g,k}^{\mathsf{rel}}(\Sym^{\epsilon}_{n,[m]}(S),0)]^{\mathrm{vir}} 
\]
denote the virtual fundamental classes associated to the relative obstruction theories. Observe that for  $\Hilb_{n}(S)$ and $\Sym_{n}(S)$, these classes coincide with the usual  virtual fundamental classes as the projection morphisms to $\mathcal{FM}(S)$ are trivial. 
\subsection{Invariants} \label{sectioninv}
 The moduli spaces of maps admit evaluation morphisms given by evaluating a map at a marked point  of the source curve, 
 \begin{align*}
 &\ev_j \colon \Mbar^{\mathsf{rel}}_{g,k}(\Hilb^{\epsilon}_{n,[m]}(S),d) \rightarrow \Hilb^{\epsilon}_{n,[m]}(S) \\
 &\ev_j \colon \Mbar^{\mathsf{rel}}_{g,k}(\Sym^{\epsilon}_{n,[m]}(S),0) \rightarrow \CI\Sym^{\epsilon}_{n,[m]}(S).
 \end{align*}
 
We also have projection morphisms to $S$, 
 \begin{align*}
 &\widetilde{\ev}_i \colon \Mbar_{g,k}^{\mathsf{rel}}(\Hilb^{\epsilon}_{n,[m]}(S),d)  \rightarrow \mathcal{FM}_{[m]}(S) \times \Mbar_{g,k}   \rightarrow S \times \Mbar_{g,k}   \\ 
  &\widetilde{\ev}_i \colon \Mbar_{g,k}^{\mathsf{rel}}(\Hilb^{\epsilon}_{n,[m]}(S),0) \rightarrow \mathcal{FM}_{[m]}(S) \times \Mbar_{g,k}  \rightarrow S \times \Mbar_{g,k} ,
 \end{align*}
where the first morphisms are the natural projection morphism associated to relative spaces of maps, and the second morphisms are given by applying contraction morphisms to the marked point $p_i$ on FM degenerations of $S$, as in Section \ref{Mapsmaster}. Using these morphisms, we define  the following invariants
 \begin{multline*}
 	\langle \lambda_1, \dots, \lambda_k \mid \Psi_1^{h_1}\gamma_1, \hdots, \Psi_m^{h_m}\gamma_m \rangle^{\sH_n,\epsilon}_{g,d} \\
 	= \int_{[\Mbar_{g,k}^{\mathsf{rel}}(\Hilb^{\epsilon}_{n,[m]}(S),d)]^\mathrm{vir}} \prod^k_{j=1} \ev_j^* \lambda_j \prod^m_{i=1}  \Psi_i^{h_i}\widetilde{\ev}^*_i\gamma_i
\end{multline*} 	
\begin{multline*}
 	\langle \lambda_1, \dots, \lambda_k \mid \Psi_1^{h_1}\gamma_1, \hdots, \Psi_m^{h_m} \gamma_m \rangle^{\sS_n,\epsilon}_{g} \\ =\int_{[\Mbar_{g,k}^{\mathsf{rel}}(\Sym^{\epsilon}_{n,[m]}(S), 0)]^\mathrm{vir}} \prod^k_{j=1}\ev_j^*\lambda_j\prod^m_{i=1}  \Psi_i^{h_i} \widetilde{\ev}^*_i\gamma_i,
 \end{multline*}
where $\Psi_i$ are $\psi$-classes on $\mathcal{FM}_{[m]}(S)$ defined as in Section \ref{FMspaces};  $\lambda_j$ and $\gamma_i$ are some cohomology classes on the corresponding spaces. 
\section{Master space}

\subsection{Overview} Let $\epsilon_0 \in \BR_{>0}$ be a number such that 
\[
n_0:=1/\epsilon_0 \in \BN,
\]
referred to as a \textit{wall}. 
 Denote $\epsilon_+$ and $\epsilon_-$ values close to $\epsilon_0$ from the right and the left, respectively,  
\[ 
\epsilon_- < \epsilon_0 < \epsilon_+. 
\]
To such $\epsilon_0$, we can associate a master space, 
\[
\BM\Hilb_{n,[m]}^{\epsilon_0}(S),
\]
defined in \cite{NHilb}, called the moduli space of $\epsilon_0$-weighted subschemes.  It is a relative moduli space of 0-dimensional subschemes over the space
  \[
  \BM\mathcal{FM}_{[m]}(S,n),
  \] 
whose definition is recalled in Section \ref{cons}. 
  
   The master space carries a $\BC^*$-action, such that both 
  $\widetilde{\Hilb}_{n,[m]}^{\epsilon_+}(S)$ and $\Hilb_{n,[m]}^{\epsilon_-}(S)$  are $\BC^*$-fixed components, 
  \[
  \widetilde{\Hilb}_{n,[m]}^{\epsilon_+}(S), \Hilb_{n,[m]}^{\epsilon_-}(S) \subset \BM\FM_{n,[m]}^{\epsilon_0}(S)^{\BC^*},
  \] 
  where 
  \[
  \widetilde{\Hilb}_{n,[m]}^{\epsilon_+}(S):=\Hilb_{n,[m]}^{\epsilon_+}(S)\times_{\mathcal{FM}_{[m]}(S,n)} \widetilde{\mathcal{FM}}_{[m]}(S,n). 
  \]

  Similarly, there is  an analogously defined master space 
 \[
 \BM\Sym_{n,[m]}^{\epsilon_0}(S),
 \]
 which relate $\Sym_{n,[m]}^{\epsilon_+}(S)$ and $\Sym_{n,[m]}^{\epsilon_-}(S)$; it is a relative moduli space of weighted points over $  \BM\mathcal{FM}_{[m]}(S,n)$. 
 Both master spaces are smooth and proper. We refer to Section \ref{Masterfixed} for the full description of their $\BC^*$-fixed components. 

In \cite{NHilb}, these spaces were used to relate tautological intersection theories of $\FM_n(S)$, $\Hilb_n(S)$,  and $\Sym_n(S)$. In this work, we will use them to compare the Gromov--Witten invariants, 
\[ 
\langle \lambda_1, \dots, \lambda_k \rangle^{\sH_n}_{g,d} \quad \text{and} \quad \langle \lambda_1, \dots, \lambda_k\rangle^{\sS_n}_{g}, \]
by applying the wall-crossings associated to the variation of $\epsilon$. 

\subsection{Entanglement and calibration} \label{cons} 
\subsubsection{Entanglement} A weighted FM degeneration is a FM degeneration with positive integers attached to each irreducible component; see \cite[Definition 10.1]{NHilb} for the precise definition. Let 
\[
\mathcal{FM}_{[m]}(S,n)
\]
be the moduli space of (not necessarily) stable marked weighted FM degenerations of total weight $n$. For a given $\epsilon_0 \in \BR_{>0}$, consider
\[ 
\CS_1 \subset \mathcal{FM}_{[m]}(S,n),
\]
the normal crossing divisor parametrising FM degenerations with at least one end component of weight $n_0$. For $h\geq 1$, let
\[
\CS_h \subseteq \CS_1
\]
be the normal crossing stratum of $\CS_1$ parametrising FM degenerations with at least $h$ end components of weight  $n_0$. Define 
\begin{equation} \label{blowups}
\widetilde{\mathcal{FM}}_{[m]}(S,n):=\mathrm{Bl}_{\tilde{\CS}_2} \hdots \mathrm{Bl}_{\CS_\ell}(\mathcal{FM}_{[m]}(S,n))
\end{equation}
to be the inductive blow-up at the strata $\CS_h$, where $\tilde{\CS}_h$ is the proper transform fo $\CS_h$, and $\ell$ is the maximal number of end components of weight $n_0$ for FM degenerations of total weight $n$. Let 
\[
\CY_h \subset \widetilde{\mathcal{FM}}_{[m]}(S,n)
\]
be the proper transform of the exceptional divisor associated to $\tilde{\CS}_h$, and 
\[
\CY^*_h :=  \CY_h \setminus \Big( \bigcup_{i< h} \CY_h \cap \CY_i  \Big). 
\]
We call $\widetilde{\mathcal{FM}}_{[m]}(S,n)$ the space of FM degenerations with \textit{entangled} end components, introduced in dimension one by Zhou \cite[Section 2]{YZ}. The terminology alludes to the fact that the blow-ups remove the $\BC^*$-automorphisms of end components, making them \textit{entangled}. Entanglement of a FM degeneration $[(W,\underline{x})]\in  \mathcal{FM}_{[m]}(S,n)$ is a closed point in the fiber of the natural projection morphism $\widetilde{\mathcal{FM}}_{[m]}(S,n) \rightarrow \mathcal{FM}_{[m]}(S,n)$ above $[(W,\underline{x})]$. 

Around a FM degeneration $[(W,\underline{x})] \in \mathcal{FM}_{[m]}(S,n)$ with $h$ end components of weight $n_0$,
\[ 
P_1, \hdots, P_h \subset W,
\]
the  local structure of such blow-ups is provided by\footnote{To be precise, one needs to multiply by affine spaces of some dimensions to account for the degrees of freedom associated to the movement of the loci at which different components are attached. Also,  translation automorphisms of end components have to be taken into account. These features, however, do not not affect the geometry of blow-up centres. }
\begin{equation} \label{local}
\CS_1=\bigcup^h_{i=1} [\{x_i=0\}/(\BC^*)^h] \subset [\BA^h/(\BC^*)^h],
\end{equation}
such that the coordinate hyperplanes $\{x_i=0\}$ of  $\BA^h$ correspond to the locus where $P_i$ does not smooth out; the stratum $\CS_{h'}$ is  given by the intersection of $h'$ hyperplanes. We say that the end components 
\[
P_{i_1},\hdots, P_{i_{h'}}  \subset W 
\]
are \textit{entangled} if the entanglement of $[(W,\underline{x})]$ is contained in the exceptional divisor of the blow-up at the proper transform of 
\[
\{x_{i_1}=\hdots =x_{i_{h'}}=0 \} \subset \CS_{h'}
\]
away from the exceptional divisors associated to the blow-ups at the proper transforms of  $\CS_{i}$ for $i<h'$; or, in other words, if it is an element of $\CY^*$.  This has the effect of reducing $(\BC^*)^{h'}$-scaling stabilisers of these end components to $\BC^*$. This remaining scaling stabiliser can be removed by \textit{calibration},  which is discussed in Section \ref{calibration}. For example, if all $h$ end components are entangled, then the entanglement is a closed point of 
\[
\CY^*_h =[\big( \p(x_1,\hdots, x_h)\setminus \bigcup\{x_i=0\} \big)/(\BC^*)^h]\cong B\BC^*.
\] 
Indeed, in this case, it must be the orbit of $[1,\hdots, 1]$ which has only one $\BC^*$-scaling stabiliser given by the diagonal of $(\BC^*)^h$. See \cite[Lemma 2.5.1]{YZ} for the global version of the analysis above.

From the above description, it is also clear  that at every step the blow-center $\tilde{\CS}_h$ is smooth, which implies that $\widetilde{\mathcal{FM}}_{[m]}(S,n)$ is smooth.   
\subsubsection{Calibration} \label{calibration}
We then define 
\[
\BM\mathcal{FM}_{[m]}(S,n):=\p(\BM\oplus \CO) \rightarrow \widetilde{\mathcal{FM}}_{[m]}(S,n),
\] 
where 
\[
\BM:=\CO(-\CS_1)
\] is the line bundle associated to the total transform of $\CS_1$. This is the space of FM degenerations with \textit{calibrated} end components. 

Given a  weighted FM degeneration
\[
[(W,\underline{x})] \in \mathcal{FM}_{[m]}(S,n)(\BC),
\]
a \textit{calibration} of $(W,\underline{x})$ is a closed point
\begin{equation} \label{cal}
[(W,\underline{x}, c)] \in \BM\mathcal{FM}_{[m]}(S,n)(\BC)
\end{equation}
 in the fiber  of the natural projection morphism $\BM\mathcal{FM}_{[m]}(S,n) \rightarrow  \mathcal{FM}_{[m]}(S,n)$ over $[(W,\underline{x})]$. If 
 \[
 [(W,\underline{x}, c)] \in \BM\mathcal{FM}_{[m]}(S,n)(\BC)\setminus  (\p(\BM)(\BC)  \cup  \p(\CO)(\BC)),
 \]
 then it  drops the dimension of the scaling stabilisers of entangled FM degenerations further by one.  Indeed, with respect to the local model (\ref{local}), this amounts to adding the data of a non-vanishing section of the ideal sheaf 
 \[
 (x_1 \cdot \hdots \cdot x_h).
 \] 
Over the locus of FM degenerations with $h$ entangled end components, calibrations of the form (\ref{cal}) are described by 
 
\begin{equation} \label{roots}
[\big(\big( \p(x_1,\hdots, x_h)\setminus \bigcup\{x_i=0\} \big)\times \BC^*\big)/(\BC^*)^h]\cong B\mu_h,
\end{equation}
 where $(\BC^*)^h$ acts by $(t_1 \cdot \hdots \cdot t_h)^{-1}$ on the second factor $\BC^*$, giving rise to the diagonal $h$-roots-of-unity stabilisers $\mu_h$.

\subsection{Construction of master spaces}

For Hilbert schemes, the master space is the open substack of the relative Hilbert scheme, 
\[ 
\BM\Hilb_{n,[m]}^{\epsilon_0}(S) \subseteq \Hilb_n(\CW/\BM\mathcal{FM}_{[m]}(S,n) ),
\]
parametrising $(W, Z, \underline{x}, c)$ subject to the following conditions:

\begin{enumerate} 

	\item $\underline{x} \cap Z=\emptyset$,
	\item $\underline{x} \subset W^{\mathrm{sm}}$ and $Z\subset W^{\mathrm{sm}}$, 
	\item for all $x \in W$, we have $\ell_x(Z)\leq n_0$,
	\item  for all end components\footnote{An end component of a FM degeneration corresponds to a vertex of valency 1 in the corresponding intersection graph, except if it is the vertex associated to $S$.  }  $\p^{2} \subset W$, such that $\underline{p} \cap \p^{2}=\emptyset$,  we have $\ell(Z|_{\p^{2}})\geq n_0$, 
	\item the group $\Aut(W, Z,\underline{x},c)$ is finite,
	\item the weights of $W$ are given by the length of the subschemes $\ell(Z)$ restricted to each component,
	\item if $[(W,\underline{x}, c)] \in \p(\BM)(\BC) \subset \BM\mathcal{FM}_{[m]}(S,n)(\BC)$, then $(W, Z, \underline{x})$ is $\epsilon_-$-weighted,
	\item if $[(W,\underline{x}, c)] \in \p(\CO)(\BC) \subset \BM\mathcal{FM}_{[m]}(S,n)(\BC)$, then $(W, Z, \underline{x})$ is $\epsilon_+$-weighted,
	\item if $(W, Z, \underline{x})$ has end components of weight $n_0$, then they contain all points $x \in W$ such that $\ell_x(Z)=n_0$. 
\end{enumerate} 	

The master space 
\[
\BM\Sym_{n,[m]}^{\epsilon_0}(S) \subseteq \CW^n:= \CW\times_{\BM\mathcal{FM}_{[m]}(S,n)} \hdots \times_{\BM\mathcal{FM}_{[m]}(S,n)} \CW
\]
is defined in the same way. Such master spaces were first introduced by Zhou \cite{YZ} in the setting of quasimaps. The higher-dimensional version involving Fulton--MacPherson spaces appeared in \cite{NuGW}. The proof of properness follows an argument that can be  applied universally to any set-up where it makes sense to define such master spaces.


\begin{theorem} \label{masterproper}
	The master spaces $\BM\Hilb_{n,[m]}^{\epsilon_0}(S)$ and $\BM\Sym_{n,[m]}^{\epsilon_0}(S)$ are proper and smooth Deligne--Mumford stacks. 
\end{theorem}	
\begin{proof}
	The relative moduli space of subschemes 
	\[
	\Hilb_n(\CW/\BM\mathcal{FM}_{[m]}(S,n) )
	\]
	is proper over $\BM\mathcal{FM}_{[m]}(S,n)$. Hence the only source of non-separatedness of $\Hilb_n(\CW/\BM\mathcal{FM}_{[m]}(S,n) )$ is  $\BM\mathcal{FM}_{[m]}(S,n)$. More precisely, for any extension of a family of calibrated FM degenerations over the spectrum of a discrete valuation ring $\Delta$,  there exists a unique extension of a family of zero-dimensional subschemes over $\Delta$. Hence, it is enough to classify the extensions of FM degenerations; there will be a unique one for which the extension of a family of zero-dimensional subschemes is contained in $\BM\Hilb_{n,[m]}^{\epsilon_0}(S)$ at the central fiber if it is contained in $\BM\Hilb_{n,[m]}^{\epsilon_0}(S)$ at the generic fiber.
	
	 The classification of extensions of FM degenerations is given by an array of rational numbers recording the intersection multiplicities with $\CS_1$ in the local model (\ref{local}).  For the precise argument, see \cite[Section 7.6]{NuGW}, replacing every appearance of maps by subschemes; the properness in the case of maps is established in \cite[Theorem 7.13]{NuGW}. In fact, in our case, the situation is even simpler, since taking extension of maps to FM degenerations is more cumbersome. In particular, the notion of non-strictly constant  components from \cite[Definition 7.5]{NuGW} does not apply here. 
	
	The smoothness of  $\BM\Hilb_{n,[m]}^{\epsilon_0}(S)$ follows from the smoothness of relative Hilbert schemes on relatively smooth surfaces. The same discussion applies to $\BM\Sym_{n,[m]}^{\epsilon_0}(S)$. 
\end{proof}	

The master spaces contain $\epsilon$-weighted spaces associated to both $\epsilon_+$ and $\epsilon_-$. In addition, since the inequality in the condition $(4)$ is non-strict, there will be tuples
\[ 
(W, Z, \underline{x}, c),
\]
such that $W$ has an end component of weight $n_0$ carrying a subscheme of length $n_0$ supported on a point. Without a calibration, such $(W, Z, \underline{x})$ has at least one-dimensional $\BC^*$-scaling automorphisms. The calibration allows to rigidify such  automorphisms, making them stable. 

There also exists a natural $\BC^*$-action on $\BM\Hilb_{n,[m]}^{\epsilon_0}(S)$ and $\BM\Sym_{n,[m]}^{\epsilon_0}(S)$ given by scaling $\p(\BM\oplus \CO)$,
\[
z \cdot \p(\BM\oplus \CO):= \p(z\cdot \BM\oplus \CO), \quad z\in \BC^*. 
\]
Next, we will describe the fixed components of this action. 
\subsection{Fixed components of master spaces} \label{Masterfixed} Although our primary interest will be in the $\BC^*$-fixed components of moduli spaces of maps to master spaces, we need to describe the $\BC^*$-fixed components of master spaces themselves to state how the Nakajima--Groinoswki classes and orbifold classes behave with respect to the localisation; see Section \ref{masterclass}.

The $\BC^*$-fixed locus of the master space takes the following form,  
\begin{equation} \label{fixed}
	\BM \Hilb^{\epsilon_0}_{n,[m]}(X)^{\BC^*}=F_- \sqcup F_+ \sqcup \coprod_{h\geq 1} F_{h}, 
\end{equation} 
such that the precise description of the components in the above decomposition is given below. The same discussion applies  to $\BM\Sym^{\epsilon_0}_{n,[m]}(S)$ verbatim. 

\subsubsection{ Fixed component $F_-$} This is the component corresponding to calibrated FM degenerations of the form
\[
[(W,\underline{x}, c)] \in \p(\BM)(\BC) \subset \BM\mathcal{FM}_{[m]}(S,n)(\BC). 
\]
By definition of the master space, the component parametrizing tuples with such calibration   admits the following identification,  
\[
F_-= \Hilb^{\epsilon_-}_{n,[m]}(S).
\]

\subsubsection{Fixed component $F_+$} This is the component corresponding to calibrated FM degenerations of the form 

\[
[(W,\underline{x}, c)] \in \p(\CO)(\BC) \subset \BM\mathcal{FM}_{[m]}(S,n)(\BC). 
\]
By definition of the master space, the fixed component admits the following identification in this case,
\[
F_+=\widetilde{\Hilb}^{\epsilon_+}_{n,[m]}(S)
\]
where 
\[\widetilde{\Hilb}^{\epsilon_+}_{n,[m]}(S):=  \Hilb^{\epsilon_+}_{n,[m]}(S) \times_{\CF\CM_{[m]}(S,n)} \widetilde{\CF\CM}_{[m]}(S,n). 
\] 

\subsubsection{Fixed components $F_{h}$} \label{componentswall} The components $F_h$  account for the difference between $F_-$ and $F_+$. 
In this case, the calibration is of the following form 
\[
[(W,\underline{x}, c)] \in \BM\mathcal{FM}_{[m]}(S,n)(\BC)\setminus  (\p(\BM)(\BC)  \cup  \p(\CO)(\BC)).
\]
For a tuple $(W, Z, \underline{x}, c)$ to be $\BC^*$-fixed in this case, the triple  $(W, Z, \underline{x})$ must have an end component of weight $n_0$ carrying a subscheme of length $n_0$ supported on a single point (which gives rise to $\BC^*$-automorphisms in the absence of the calibration). Moreover, the entangled components are precisely those that  carry a subscheme of length $n_0$ supported on a point. The subscript $h$ in $F_h$ denotes the number of such end components. 

\begin{lemma} \label{closed}
	Assume  
	\[
	[(W,\underline{x}, c)] \in \BM\mathcal{FM}_{[m]}(S,n)(\BC)\setminus  (\p(\BM)(\BC)  \cup  \p(\CO)(\BC)),
	\]
	then a tuple 
	\[
	(W, Z, \underline{x}, c)\in \BM\Hilb_{n,[m]}^{\epsilon_0}(S)
	\]
	 is $\BC^*$-fixed if and only if (1) $W$ has an end component $P$ of  weight $n_0$, such that $Z|_{P}$ is of length $n_0$ supported on a single point which is fixed by the $\BC^*$-scaling of the end component;\footnote{In local coordinates, this is equivalent to the homogeneity of the defining ideal of the subscheme.} (2) all such end components are entangled. 
\end{lemma}	

\begin{proof} By the local description (\ref{local}), in order for the tuple to be $\BC^*$- fixed, $(W, Z, \underline{x})$ must have $\BC^*$-automorphisms at end components; this happens precisely when these components are of weight $n_0$ and carry a subscheme of length $n_0$ fixed by the $\BC^*$-scaling. If  at least one of such end components is not entangled, then   $(W, Z, \underline{x}, c)$ has infinitely many automorphisms, and therefore are unstable. If there are more entangled components than components  of weight $n_0$ carrying a subscheme of length $n_0$, then $(W, Z, \underline{x}, c)$ cannot be fixed because the $\BC^*$-scaling will move subschemes on end components of $(W, Z, \underline{x})$ with no stabilisers. For more details, we refer to the proof of the similar statement \cite[Lemma 6.5.1]{YZ} in the case of quasimaps.
\end{proof}	




To give a global description of the fixed components $F_h$, consider  the gluing morphism given by attaching the entangled end components of weight $n_0$,
\[
\widetilde{\mathrm{gl}}_h \colon \widetilde{\CF\CM}_{[m+h]}(S,n')\times_{S^h} \CB^h \rightarrow \tilde{\CS}_{h} \subset  \mathrm{Bl}_{\tilde{\CS}_{h-1}}\hdots \mathrm{Bl}_{\CS_\ell}(\CF\CM_{[m]}(S,n)),
\]
induced by the gluing morphism for $\CF\CM_{[m]}(S,n)$ and  the compatibility of blow-ups on both sides; see \cite[Lemma 2.4.2]{YZ}.
Here, 
\[
\CB:=BG_S=[S/G_S] \rightarrow S 
\] 
is the $S$-relative stack of end components, such that  
\[ G_S\subset \Aut_S(\p (T_S \oplus \CO))\]
be the group $S$-scheme of $S$-relative automorphisms of $\p (T_S \oplus \CO)$, fixing $ \p(T_S) \subset \p (T_S \oplus \CO))$ pointwise. The stack $BG_S$ carries a universal end component $\CP$ together with a divisor at infinity $\CD$
\[
\CD=[\p (T_S)/G_S] \hookrightarrow \CP= [\p (T_S \oplus \CO)/G_S]\rightarrow \CB.
\]
The morphism $\widetilde{\mathrm{gl}}_h$ attaches end components to FM degenerations, such that the divisors $\CD$ is glued to the exceptional divisors
\[
\CE_i \subset \mathrm{Bl}_{x_{m+1},\hdots, x_{m+h}}(\CW), \quad i=1,\hdots, h,
\]
associated to the blow-up at the marked point $x_i$ of FM degenerations. 

We define the cotangent line bundle associated to $\CD$ by 
\[
\BL(D)^\vee:=N_{\CD} \otimes \CO_{\p(T_{S})}(-1),
\]
were $\CO_{\p(T_{S})}(-1)$ is pulled back by the natural $S$-relative  morphism 
\[
\CD \rightarrow \p(T_{S}). 
\]
On the other, the cotangent line bundles associated to $\CE_i$ are defined by 
\[ 
\BL(E_i)^\vee:=N_{\CE_i} \otimes \CO_{\p(T_{S})}(1),
\] 
where $\CO_{\p(T_{S})}(1)$  is pulled back via the $S$-relative forgetful morphism associated to $x_i$, 
\[
\mathrm{Bl}_{x_{m+1},\hdots, x_{m+h}}(\CW) \rightarrow \p(T_S).
\]
We refer to \cite[Section 9]{NHilb} for more details about the line bundles $\BL(E_i)$ - they are analogues of cotangent line bundles associated to markings on curves. 

Finally, let 
\[
\widetilde{\Hilb}^{\epsilon_+}_{n',[m+h]}(S)^{\frac{1}{h}}
\]
be the stack of $h$-roots of the calibration bundle $\BM_{\widetilde{\CF\CM}_{[m+h]}(S,n')}$ pulled back from $\widetilde{\CF\CM}_{[m]}(S,n')$. 

\begin{proposition} \label{important}  There exists a natural isomorphism, 
	\[
	\widetilde{\mathrm{gl}}^*_hF_h\cong  \widetilde{\Hilb}^{\epsilon_+}_{n',[m+h]}(S)^{\frac{1}{h}} \times_{S^h} \prod_{i=1}^{h}  \Hilb_{n_0}(T_S/S)^{\BC^*},
	\]
	where the $\BC^*$-action on the right is given by scaling of the tangent bundle $T_S$. 
	
\end{proposition}

\begin{proof} First, by Lemma \ref{closed}, the projection 
	\[
	F_h \rightarrow \widetilde{\CF\CM}_{[m]}(S,n) \rightarrow \mathrm{Bl}_{\tilde{\CS}_{h-1}} \hdots \mathrm{Bl}_{\CS_\ell}(\CF\CM_{[m]}(S,n)),
	\]
	given by forgetting the entanglement and calibration, factors through 
	\[
	\tilde{\CS}_{h} \subset \mathrm{Bl}_{\tilde{\CS}_{h-1}} \hdots \mathrm{Bl}_{\CS_\ell}(\CF\CM_{[m]}(S,n)),
	\] 
		To establish the claim, we will trivialize the entangled end components of FM degenerations in the fiber product	$\widetilde{\mathrm{gl}}^*_h F_h$ defined by 
	\begin{equation*}
	\begin{tikzcd}[row sep = scriptsize, column sep = scriptsize]
		& \widetilde{\mathrm{gl}}^*_h F_h  \arrow[d]  \arrow[r] &  F_h \arrow[d] \\
		& \widetilde{\CF\CM}_{[m+h]}(S,n')\times_{S^h} \CB^h  \arrow[r, "\widetilde{\mathrm{gl}}_h "] &  \tilde{\CS}_{h}.
	\end{tikzcd}
\end{equation*}
	Note that taking the pullback with respect to $\widetilde{\mathrm{gl}}_h$ corresponds to ordering the entangled end components. 
	
	Let 
	\[
	Y:=(\mathrm{Tot}(\BL(D_1)^\vee)\setminus \{0\}) \times_{\widetilde{\mathrm{gl}}^*_h F_h}  \hdots \times_{\widetilde{\mathrm{gl}}^*_h F_h} (\mathrm{Tot}(\BL(D_h)^\vee)\setminus \{0\})
	\]
	be  the total space of the tangent line bundles $\BL(D_i)^\vee$ without the zero sections  associated to the entangled end components,
	\[
	\CD_i \hookrightarrow \CP_i \rightarrow \CB, \quad i=1,\hdots, h. 
	\]
	Here, $\CB$ is the moduli space of the $i$-th  end components in 
	\[
	\widetilde{\CF\CM}_{[m+h]}(S,n')\times_{S^h} \CB^h,
	\]
	and $\CD_i$  is the divisor at infinity of the $i$-th end component. 
	
	Next, by Lemma \ref{closed},   each entangled end  component in $F_h$ carries a $0$-dimensional subscheme supported on a single point. Using the support of such subschemes,  over $Y$ the fibers of $\CP_i$ can be identified with $\p(T_S\oplus \CO)$ (or, in other words, trivialised),
	\[
	\CP_i \times_{\widetilde{\mathrm{gl}}^*_h F_h} Y \rightarrow  \p(T_S\oplus \CO)
	\] by sending vectors from $(\mathrm{Tot}(\BL(D_i)^\vee)\setminus \{0\})$  to some fixed vector 
	\[
	v_\infty \in H^0(\p(T_S), N_{\p(T_S)} \otimes \CO_{\p(T_S}(-1));
	\]  the unique point which supports the 0-dimensional subscheme to the origin $0=[0,1]\in \p(T_S\oplus \CO)$; and $\CD_i$ to $\p(T_S)$ (note that $\CD_i$ can be canonically identified with $\p(T_S)$). This in turn induces a morphism 
	\begin{equation}\label{mor1}
	Y \rightarrow \Hilb_{n_0}(T_S/S),
	\end{equation}
where $T_S \subset \p(T_S\oplus \CO)$ is the complement of the divisor at infinity $\p(T_S)$; since subschemes cannot lie on the singular loci of FM degenerations, subschemes are supported away from $\p(T_S)$. 

The calibrations in $F_h$ are of the form,
\[
[(W,\underline{x}, c)] \in \BM\mathcal{FM}_{[m]}(S,n)(\BC)\setminus  (\p(\BM)(\BC)  \cup  \p(\CO)(\BC)),
\]
hence, over $\widetilde{\mathrm{gl}}^*_h F_h$, we have 
\begin{equation} \label{trivial}
\BM_{\widetilde{\CF\CM}_{[m]}(S,n)} \cong \CO_{\widetilde{\mathrm{gl}}^*_h F_h}.
\end{equation}
Moreover, there is a natural identification, 
\begin{equation} \label{}
\BM_{\widetilde{\CF\CM}_{[m]}(S,n)}\cong \BM_{\widetilde{\CF\CM}_{[m+h]}(S,n')} \otimes \Theta^\vee_1 \otimes \hdots \otimes \Theta^\vee_h,
\end{equation}
where, over a closed point, 
\[
\Theta^\vee_i:=\BL(E_{m+i}) \otimes \BL(D_i)
\]is identified with the dual of the space of infinitesimal smoothings of the $i$-th end component.  
By \cite[Lemma 2.5.5]{YZ}, the entanglement of end components in $ F_h$ also gives a natural identification,
\[
\Theta_1\cong  \Theta_i.
\]
Overall, we obtain 
	\begin{equation*}
		\Theta_1^{\otimes h} \cong\BM_{\widetilde{\CF\CM}_{[m+h]}(S,n')},
	\end{equation*}
i.e., there is a natural $h$-root of $\BM_{\widetilde{\CF\CM}_{[m+h]}(S,n')}$ on $\widetilde{\mathrm{gl}}^*_h F_h$. In particular, this gives rise to a morphism
\begin{equation} \label{mor2}
Y \rightarrow \widetilde{\Hilb}^{\epsilon_+}_{n',[m+h]}(S)^{\frac{1}{h}}.
\end{equation}
Combining (\ref{mor1}) and (\ref{mor2}), we obtain  a morphism, 
	\begin{equation} \label{morphismY}
		Y \rightarrow \widetilde{\Hilb}^{\epsilon_+}_{n',[m+h]}(S)^{\frac{1}{h}} \times_{S^h} \prod_{i=1}^{h}  \Hilb_{n_0}(T_S/S),
	\end{equation}
where the fiber product over $S^h$ is due to the gluing of the end components; the $h$-roots are as in (\ref{roots}). 
	
	The space $Y$ admits a $\BC^* \times (\BC^*)^h$-action, such that the first factor $\BC^*$ comes form the scaling action on the master space, while $(\BC^*)^h$ scales $\BL(D_i)^\vee$. The space on the right of (\ref{morphismY}) also admits a $(\BC^*)^h$-action given by scaling each factor  $\Hilb_{n_0}(T_S/S)$. The two actions are related by 
	\begin{align}\label{torusaction}
		\begin{split}
			\BC^*\times (\BC^*)^h &\rightarrow (\BC^*)^h \\
			(z, t_1, \hdots, t_h) &\mapsto (z^{1/h}t_1^{-1},\hdots, z^{1/h}t_h^{-1}),
		\end{split}
	\end{align}
	the detailed explanation of the relation between torus weights of two spaces is given in the proof of \cite[Lemmas 6.5.4, 6.5.6]{YZ}.  The quotient of $Y$ by  $(\BC^*)^h$ removes  $\mathrm{Tot}(\BL(D_i)^\vee)\setminus \{0\}$, hence, 
	\[ 
	Y/(\BC^*)^h\cong \widetilde{\mathrm{gl}}^*_h F_h. 
	\]
	Moreover, the morphism (\ref{morphismY}) lands into $(\BC^*)^h$-fixed locus, since the support of subschemes is sent to the origin. Hence, the morphism (\ref{morphismY})  descends to
	\begin{equation} \label{morphismdes}
		\widetilde{\mathrm{gl}}^*_hF_h \rightarrow \widetilde{\Hilb}^{\epsilon_+}_{n',[m+h]}(S)^{\frac{1}{h}} \times_{S^h} \prod_{i=1}^{h}  \Hilb_{n_0}(T_S/S)^{\BC^*}.
	\end{equation}
	
By inspecting the deformation theory on both sides, it is not difficult to see that this morphism is in fact an isomorphism.  A more direct argument from \cite[Lemma 6.5.5]{YZ}, which applies to our case verbatim, allows to construct the inverse
	\begin{equation} \label{inverse}
	\widetilde{\Hilb}^{\epsilon_+}_{n',[m+h]}(S)^{\frac{1}{h}} \times_{S^h} \prod_{i=1}^{h}  \Hilb_{n_0}(T_S/S) \rightarrow  \widetilde{\mathrm{gl}}^*_hF_h.
	\end{equation}
	This is achieved by furnishing the FM degenerations  with entanglements and calibrations. We will sketch the argument. First, there is a projection
	\[
	\widetilde{\Hilb}^{\epsilon_+}_{n',[m+h]}(S)^{\frac{1}{h}} \times_{S^h} \prod_{i=1}^{h}  \Hilb_{n_0}(T_S/S) \rightarrow  \widetilde{\CF\CM}_{[m+h]}(S,n')\times_{S^h} \CB^h.
	\]
In particular, in order to construct the inverse morphism, we have to lift this projection to  $\widetilde{\mathrm{gl}}_h^*\CY^*_h$, which, by \cite[Lemma 2.5.1]{YZ}, is isomorphic to 
\[
\p(\Theta_1\oplus \hdots \oplus \Theta_h ) \setminus \Big(\bigcup_{i} \p(\Theta_1\oplus \hdots \hat{\Theta}_i \hdots \oplus \Theta_h \Big) .
\]
This will give the entanglement. The calibration is then constructed by picking a non-vanishing section of 
\[ 
\BM_{\widetilde{\CF\CM}_{[m+h]}(S,n')} \otimes \Theta^\vee_1 \otimes \hdots \otimes \Theta^\vee_h. 
\]
In both cases, the construction involves passing to a space where the line bundles in question trivialise canonically  and then descending the construction from that space (as we did it with $Y$ for (\ref{morphismdes})). 
The existence of such lifts  can also be easily seen with the local model (\ref{local}).
\end{proof}


Observe that by the construction of $Y$ as the total space without the zero section of $\BL(D_i)^\vee$ and  the definition of the $(\BC^*)^h$-action, the weights $t_i$ of $(\BC^*)^h$ become 
\[
-\psi(D_i):=\mathrm{c}_1(\BL(D_i)^\vee)
\] after identifying  the quotient of $Y$ by $(\BC^*)^h$ with $\widetilde{\mathrm{gl}}^*_h F_h$.  By the relation between the torus actions (\ref{torusaction}), the weight $z$ of $\BC^*$ acting on the master space therefore becomes 
\begin{equation} \label{weighttwist}
	z/h+\psi(D_i)
\end{equation}
on each factor $  \Hilb_{n_0}(T_S/S)^{\BC^*}$. 

 \section{Classes on master spaces} \label{sectionmaster}
 
 \subsection{Nakajima--Grojnowski classes on master spaces} \label{Nakajimamaster}
 As in Section \ref{epsilonNak}, we can similarly define Nakajima--Grojnowski classes on the master space $\BM \Hilb^{\epsilon_0}_{n,[m]}(S)$ by using   $\BM\mathcal{FM}_{[m]}(S,n)$-relative nested Hilbert schemes. 
 
 For a cohomologically weighted partition $\mu$, let 
 \[ 
 \BM \Hilb^{\epsilon_0}_{\mu,[m]}(S)
 \]
 be the $  \BM\mathcal{FM}_{[m]}(S,n)$-relative moduli space of flags of $0$-dimensional subschemes, 
 \[
 \underline{Z}=\{ \emptyset\subset  Z_1 \subset Z_2  \subset \hdots \subset Z_{\ell(\mu)} \},
 \]
 such that $Z_{\ell(\mu)} \in \BM \Hilb^{\epsilon_0}_{n,[m]}(S)(\BC)$, and \[ 
 \CO_{Z_{i}}/\CO_{Z_{i-1}}
 \]
 is a length $\mu_i$ sheaf supported on a single point. Using the correspondence, 
 \begin{equation} \label{correspondence}
 	\begin{tikzcd} [row sep = small, column sep = small]
 		&  &  \BM \Hilb^{\epsilon_0}_{\mu,[m]}(S) \arrow[dl,"q"'] \arrow[dr, "p"] &\\
 		& S^{\ell(\mu)}  &  &\BM \Hilb^{\epsilon_0}_{n,[m]}(S)
 	\end{tikzcd}
 \end{equation}
 we define the Nakajima--Grojnowski classes 
 \begin{equation*}
 	\lambda_{\sH}(\mu):=\frac{1}{\mathfrak{z}(\mu)} p_*(q^*(\gamma_{i_1}\otimes \hdots \otimes \gamma_{i_{\ell(\mu)}})\cap [\BM \Hilb^{\epsilon_0}_{\mu,[m]}(S)] ) \in H^*(\BM \Hilb^{\epsilon_0}_{n,[m]}(S)). 
 \end{equation*}
 As before, the exceptional class is given by 
 \begin{equation}\label{masterclass}
 	E=\mathfrak{z}(\mu) \lambda_{\sH}(\mu) \quad \text{for} \quad \mu= ((2,[\mathrm{pt})],(1,[\mathrm{pt})], \hdots, (1, [\mathrm{pt}])).
 \end{equation}
 Note that this is a curve class only relatively to  $\BM\mathcal{FM}_{[m]}(S,n)$. 
 \subsection{Orbifold classes on master spaces} Orbifold classes can also be defined on $\BM\Sym^{\epsilon_0}_{n,[m]}(S)$ using the $\BM\mathcal{FM}_{[m]}(S,n)$-relative inertia stacks. Consider the $\BM\mathcal{FM}_{[m]}(S,n)$-relative inertia stack of $\BM\Sym^{\epsilon_0}_{n,[m]}(S)$, 
 
 \begin{equation*}
 	\begin{tikzcd}[row sep = scriptsize, column sep = scriptsize]
 		&\CI\BM\Sym^{\epsilon_0}_{n,[m]}(S) \arrow[d]  \arrow[r] &   \BM\Sym^{\epsilon_0}_{n,[m]}(S)  \arrow[d,"\Delta"] \\
 		&  \BM\Sym^{\epsilon_0}_{n,[m]}(S) \arrow[r,"\Delta"]&   \BM\Sym^{\epsilon_0}_{n,[m]}(S)\times_{\BM\mathcal{FM}_{[m]}(S,n)} \BM\Sym^{\epsilon_0}_{n,[m]}(S),
 	\end{tikzcd}
 \end{equation*}
 where 
 \[ 
 \Delta\colon  \BM\Sym^{\epsilon_0}_{n,[m]}(S) \rightarrow \BM\Sym^{\epsilon_0}_{n,[m]}(S)\times_{\BM\mathcal{FM}_{[m]}(S,n)} \BM\Sym^{\epsilon_0}_{n,[m]}(S)
 \]
 is the $\BM\mathcal{FM}_{[m]}(S)$-relative diagonal. We have 
 \[ 
 \CI\BM\Sym^{\epsilon_0}_{n,[m]}(S)= \coprod_{[\sigma]} [\BM\Sym^{\epsilon_0}_{n,[m]}(S)^\sigma/C(\sigma)],
 \]
 where the fixed locus associated to a group element $\sigma$ is defined as in (\ref{fixedlocus}) with respect to the projection morphism 
 \[
 p\colon \BM\Sym^{\epsilon_0}_{n,[m]}(S) \rightarrow \BM\mathcal{FM}_{[m]}(S,n). 
 \] 
 We define the orbifold classes on master spaces as follows, 
 \begin{multline*}
 	\lambda_{\sS}(\mu)=\frac{1}{\mathfrak{z}(\mu)} \tau^* \bigg(\sum_{\rho\in C(\sigma)}\rho^*(\gamma_{i_1} \otimes\hdots \otimes \gamma_{i_\ell} ) \bigg) \\
 	\in H^{*-2\mathrm{age}(\sigma)}([\BM\Sym^{\epsilon}_{n,[m]}(S)^\sigma/C(\sigma)],\BC),
 \end{multline*}
 where 
 \[ 
 \tau \colon \BM\Sym^{\epsilon}_{n,[m]}(S)^\sigma \rightarrow S^{\ell(\sigma)}
 \]
 is the contraction morphism which commutes with the action of $C(\sigma)$.

 \subsection{Splitting of Nakajima--Grojnowski classes}
 
Since the $\BC^*$-action on $\BM \Hilb^{\epsilon_0}_{n,[m]}(S))$ comes from the base $\BM\mathcal{FM}_{[m]}(S,n)$, and the classes 	$\lambda_{\sH}(\mu)$ are defined via $\BM\mathcal{FM}_{[m]}(S,n)$-relative spaces of flags, these classes are naturally $\BC^*$-equivariant
\[
\lambda_{\sH}(\mu) \in H_{\BC^*}^*(\BM \Hilb^{\epsilon_0}_{n,[m]}(S)).
\]
In this section, we will determine how $\lambda_{\sH}(\mu)$ restrict to the $\BC^*$-fixed loci of $\BM \Hilb^{\epsilon_0}_{n,[m]}(S)$, considered in Section \ref{componentswall}. 

First, since $F_-$ and $F_+$ are simply  

\[
\BM \Hilb^{\epsilon_0}_{n,[m]}(S))|_{\p(\CO(-S_1))} \quad \text{and} \quad  \BM \Hilb^{\epsilon_0}_{n,[m]}(S))|_{\p(\CO)},
\]
we obtain that $\lambda_{\sH}(\mu)$ restrict to the same Nakajima--Grojnowski classes, defined in Section \ref{epsilonNak}, with respect to the identifications from Section \ref{Masterfixed}.

For the components $F_h$,  by Lemma \ref{important}, we have to determine how the restriction of $\lambda_{\sH}(\mu)$  splits on the product 

\begin{equation} \label{themap}
 \widetilde{\Hilb}^{\epsilon_+}_{n',[m+h]}(S)^{\frac{1}{h}} \times_{S^h} \prod_{i=1}^{h}  \Hilb_{n_0}(T_S/S)^{\BC^*}.
\end{equation}
If $\gamma_{i_j}=\mathbb{1}$ for all parts in a weighted partition $\mu$, then the class $\lambda_{\sH}(\mu)$ can be represented by  the locus of $0$-dimensional subschemes, whose multiplicity profile is given by the partition $(\mu_1,\hdots,\mu_{\ell(\mu)})$. Hence, on the product above, these multiplicities are simply distributed among its  factors. Next lemma makes this precise.
 \begin{lemma} \label{Nsplitting} With respect to the identification from Proposition \ref{important}, we have
 	\[
 	\lambda_{\sH}(\mu)|_{ \widetilde{\mathrm{gl}}_h^*F_h}=\sum_{(\mu^0, \mu^1, \hdots \mu^h)} \lambda_{\mathsf{H}}(\mu^0)\cdot \prod^h_{i=1} \lambda_{\mathsf{H}}(\mu^i)(z/h+\psi(D_i)),
 	\]
 	such that the sum is taken over all ordered decompositions of $\mu$,\footnote{More explicitly, for the partition $((\mu_1, \gamma_1), ((\mu_2, \gamma_2)))$, decompositions $\{\mu^1=(\mu_1, \gamma_1), \mu^2= (\mu_2, \gamma_2)\}$ and $\{\mu^1=(\mu_2, \gamma_2), \mu^2= (\mu_1, \gamma_1)\}$ are distinct. } 
 	\[
 	\mu=\mu^0\sqcup \mu^1 \sqcup \hdots \sqcup \mu^h,
 	\]
 	into pairwise distinct and possibly empty subpartitions $\mu^i \subseteq \mu$, such that 
 	\[
  |\mu^0|=n'=n-hn_0  \quad  \text{and} \quad	|\mu^i|=n_0 \quad \text{for} \quad i=1,\hdots h. 
 	\]
 	The classes  
 	\[
 	\lambda_{\sH}(\mu^0) \in H^*(\widetilde{\Hilb}^{\epsilon_+}_{n',[m+h]}(S)^{\frac{1}{h}})
 	\]
 	 are Nakajima--Grojnowski classes on $\widetilde{\Hilb}^{\epsilon_+}_{n',[m+h]}(S)$ pulled back to the stack of $h$-roots. While the classes  
 	\[
 	\lambda_{\sH}(\mu^i)(z) \in H^*(\Hilb_{n_0}(T_S/S)^{\BC^*})[z^\pm]
 	\]
 	 are the $\BC^*$-equivariant Nakajima--Grojnowski on the fixed locus of $\Hilb_{n_0}(T_S/S)$. 
 	
 \end{lemma}
 \begin{proof} 
 	
 First, consider the following pull-back diagram 
	\begin{equation*}
	\begin{tikzcd}[row sep = scriptsize, column sep = scriptsize]
		& \widetilde{\mathrm{gl}}^*_h \BM\Hilb_{n,[m]}^{\epsilon_0}(S)  \arrow[d]  \arrow[r,"\iota"] &  \BM\Hilb_{n,[m]}^{\epsilon_0}(S)  \arrow[d] \\
		& \widetilde{\CF\CM}_{[m+h]}(S,n')\times_{S^h} \CB^h  \arrow[r, "\widetilde{\mathrm{gl}}_h "] &  \mathrm{Bl}_{\tilde{\CS}_{h-1}} \hdots \mathrm{Bl}_{\CS_\ell}(\CF\CM_{[m]}(S,n)).
	\end{tikzcd}
\end{equation*}
The restriction $\iota^!(\lambda_{\sH}(\mu))$ is given by the Nakajima--Grojnowski classes associated to the pullback of the nested Hilbert scheme
\[
\widetilde{\mathrm{gl}}^*_h \BM \Hilb^{\epsilon_0}_{\mu,[m]}(S).
\]
Next, by the argument from end of the proof of Proposition \ref{important},  we have a natural morphism 
\begin{equation} \label{smooth}
\pi \colon \widetilde{\Hilb}^{\epsilon_+}_{n',[m+h]}(S)^{\frac{1}{h}} \times_{S^h} \prod_{i=1}^{h}  \Hilb_{n_0}(T_S/S)\rightarrow \widetilde{\mathrm{gl}}^*_h \BM \Hilb^{\epsilon_0}_{n,[m]}(S).
\end{equation}
The comparison of the first-order deformations of tuples $(W, Z, \underline{x}, c)$ on both sides reveals that the morphism is smooth, such that the fibers are given by the orbits of the translation action of $T_S$ on each factor  $\Hilb_{n_0}(T_S/S)$ since the end components in the space on right are considered up to translations.\footnote{Note that the morphism is not surjective  because there is no condition on the entanglement on the right-hand side. } Moreover, it respects the $\BC^*$-actions on both sides as in (\ref{torusaction}), and the restriction of $p$  to the fixed locus gives the isomorphism onto $\widetilde{\mathrm{gl}}^*_h F_h$ from Proposition \ref{important} (note that the translation deformations lie in the moving part of the deformation space). Overall, in order to determine how $\lambda_{\sH}(\mu)$ restricts to $\widetilde{\mathrm{gl}}^*_h F_h$, we therefore have to determine how $\pi^!(\iota^!(\lambda_{\sH}(\mu)))$ restricts to $\widetilde{\mathrm{gl}}^*_h F_h$. Now, the class $\pi^!(\iota^!(\lambda_{\sH}(\mu)))$ is given by the pull-back of the relative nested Hilbert with respect to $\pi$ and the morphism (\ref{topterm}),
\[
N_{\mu,h}:=\pi^*(\widetilde{\mathrm{gl}}^*_h \BM \Hilb^{\epsilon_0}_{\mu,[m]}(S)).
\]
More explicitly, it parametrises flags $\underline{Z}$ whose top element $Z_{\ell(\mu)}$ lies in 

\begin{equation}\label{prodh} \widetilde{\Hilb}^{\epsilon_+}_{n',[m+h]}(S)^{\frac{1}{h}} \times_{S^h} \prod_{i=1}^{h}  \Hilb_{n_0}(T_S/S). 
\end{equation}

 	 Consider a subscheme $Z$ that belongs to (\ref{prodh}). It  splits into a disjoint union of $0$-dimensional subschemes contained in each of the factors of (\ref{prodh}),  
 	\[ 
 Z= Z_{0,\ell(\mu) } \sqcup Z_{1,\ell(\mu) } \sqcup \hdots \sqcup  Z_{h,\ell(\mu) }.
 	\]
 	Hence, the nested Hilbert scheme $N_{\mu,h}$ decomposes accordingly: since a sheaf $\CO_{Z_i}/\CO_{Z_{i-1}}$ is required to be supported on a single point, a flag $\underline{Z}$ splits among the components of $Z$. More precisely, for every decomposition of the partition $\mu $, 
 	\[ 
 	\mu=\mu^0\sqcup \mu^1 \sqcup \hdots \sqcup \mu^h, \quad |\mu^i|=n_0 \text{ if }i\geq 1,
 	\]
 	there is an open-closed subspace of $N_{\mu,h}$, denoted by $N_{\mu^0, \hdots, \mu^h}$. It  parametrises flags whose associated graded sheaf has the  multiplicity profile specified by $\mu^i$ on each $i$-th component of (\ref{prodh}), 
 	\begin{equation} \label{decomp}
 N_{\mu,h}= \coprod_{(\mu^0, \hdots, \mu^h)} N_{\mu^0, \hdots, \mu^h}, 
 	\end{equation}
 	such that if $|\mu^i |$ is not equal to the length of points on the corresponding component for some $i$, the space $N_{\mu^0, \hdots, \mu^h}$ is empty. The order of elements of flags still is required to respect the standard order of the weighted partition $\mu$.\footnote{There is another less-refined way to label these open-closed subspaces, namely by unweighted decompositions of the partition $(\mu_1 \hdots,\mu_{\ell(\mu)})$. The result is the same, as these open-closed subspaces will simply contain more connected components. }
 	
 	 Each component  $N_{\mu^0, \hdots, \mu^h}$ can in turn be realised as several copies of the product of nested Hilbert schemes, 
 	\begin{equation} \label{productsmultiple}
 	\widetilde{\Hilb}^{\epsilon_+}_{\mu^0,[m+h]}(S)^{\frac{1}{h}} \times_{S^h} \prod_{i=1}^{h} \Hilb_{\mu^i, n_0}(T_S/S),
 	\end{equation}
 	of which there are  
 	\[
 	\frac{|\Aut(\mu)|}{\prod_i |\Aut(\mu^i)|},
 	\] 
 	arising from the choices of assembling a flag associated to $\mu$ from smaller flags  associated to $\mu^i$, respecting the standard order of $\mu$. 
 	The pushorward of the fundamental class of this product to (\ref{prodh}) is the product of pushforwards of fundamental classes of 
 	\[
 	\widetilde{\Hilb}^{\epsilon_+}_{\mu^0,[m+h]}(S)  \times_{S^h} \prod_{i=1}^{h} \Hilb_{n_0}(T_S/S),
 	\] 
 	and
 	\[
 	\widetilde{\Hilb}^{\epsilon_+}_{n',[m+h]}(S) \times_S \Hilb_{\mu^\ell, n_0}(T_S/S) \times_{S^{h-1}} \prod_{i\neq \ell} \Hilb_{ n_0}(T_S/S). 
 	\]  
 	After multiplying by $1/\mathfrak{z}(\mu)$, this is  equal to  
 	\[ 
 	\lambda_{\mathsf{H}}(\mu^0)\otimes \bigotimes^h_{i=1} \lambda_{\mathsf{H}}(\mu^i)(z),
 	\]
 	such that $|\Aut(\mu)|$ accounts for multiple copies of the product (\ref{productsmultiple}),\footnote{In other words, $1/|\Aut(\mu)|$ removes the order on the support of the subschemes, without which, on different factors of (\ref{productsmultiple}), the classes split without any combinatorial factors. } while $\prod \mu_j$ splits among  the classes $\lambda_{\mathsf{H}}(\mu^i)$. 
 	Summing over components in (\ref{decomp}) and substituting (\ref{weighttwist}), we  obtain the statement of the lemma. 
\end{proof}

\subsection{Splitting of orbifold classes} 
The restrictions of orbifold classes satisfy the same splitting formula as restrictions of Nakajima--Grojnowski classes in Lemma \ref{Nsplitting}. In the next lemma, $F_h$ denotes an analogously defined fixed component in the master space $\BM\Sym^{\epsilon_0}_{n,[m]}(S)$. 
\begin{lemma} With respect to the identification from Proposition \ref{important} applied to symmetric products,  we have
\[
\lambda_{\mathsf{S}}(\mu)|_{\widetilde{\mathrm{gl}}_h^*F_h}=\sum_{(\mu^0, \mu^1, \hdots \mu^h)} \lambda_{\mathsf{S}}(\mu^0)\otimes \bigotimes^h_{i=1} \lambda_{\mathsf{S}}(\mu^i)(z/h+\psi(D_i)),
\]
where we use the same notation as in Lemma \ref{Nsplitting}. 
\end{lemma}

\begin{proof}

By the same argument as in the case of Hilbert schemes in Lemma \ref{Nsplitting}, it is enough to consider the splitting of orbifold classes on products

 	\begin{equation}  \label{productsym} \widetilde{\Sym}^{\epsilon_+}_{n',[m+h]}(S)^{\frac{1}{h}} \times_{S^h} \prod_{i=1}^{h}  \Sym_{n_0}(T_S/S).
\end{equation}
To be more precise, before taking the quotient by the $\Sigma_n$-action, we have a disjoint union of components
\begin{equation} \coprod_{ (I_0,\hdots, I_h)} \widetilde{\FM}^{\epsilon_+}_{[n'+m+h]}(S)^{\frac{1}{h}} \times_{S^h} \prod_{i=1}^{h} (T_S/S)^{n_0},
\end{equation}
where the union is taken over  partitions $(I_0,\hdots, I_h)$  of the set $\{1,\hdots,n\}$, such that $|I_i|=n_0$ for $i\geq 1$. The symmetric group $\Sigma_n$ acts by permuting the points within a component but also between different components, depending on a permutation; its quotient is (\ref{productsym}).  
The orbifold class $\lambda_{\mathsf{S}}(\mu)$ is defined via the locus fixed by an element $\sigma \in \Sigma_n$ in the conjugacy class associated to the partition $\mu$ of $n$, 
\begin{equation} \label{masterfixed}
 \left(\widetilde{\Sym}^{\epsilon_+}_{n',[m+h]}(S)^{\frac{1}{h}} \times_{S^h} \prod_{i=1}^{h}  \Sym_{n_0}(T_S/S)\right)^\sigma, 
\end{equation}
which in turn is defined as in (\ref{fixedlocus}). We can represent $\sigma$ in terms of disjoint cyclic permutations
\[ 
\sigma=\sigma_{1} \cdot \hdots \cdot \sigma_{\ell(\mu)},
\] 
such that $\sigma_j$ is a cycle of length $\mu_j$. For a decomposition of the partition $(\mu_1,\hdots, \mu_{\ell(\mu)})$ into (non-weighted) subpartitions  
\[ 
\tilde{\mu}^0\sqcup \tilde{\mu}^1 \sqcup \hdots \sqcup \tilde{\mu}^h,  \quad |\tilde{\mu}^i|=n_0 \text{ if }i\geq 1.
\]
we can decompose $\sigma$ into subpermutations $\sigma(i)$ of type $\tilde{\mu}^i$ composed of the cycles $\sigma_j$,
\[
\sigma=\sigma(0) \cdot \hdots \cdot \sigma(h).
\]
In general, such a decomposition in not unique - there are 
\[
\frac{|\Aut(\mu_1,\hdots, \mu_{\ell(\mu)})|}{\prod_i |\Aut(\tilde{\mu}^i)|}
\]
cycle decompositions associated to a partition decomposition. Then, the fixed locus (\ref{masterfixed}) decomposes accordingly
\begin{equation} \label{masterfixed2}
	(\ref{masterfixed})= \coprod_{(\sigma(0),  \hdots , \sigma(h) )}  \left(	\widetilde{\Sym}^{\epsilon_+}_{n',[m+h]}(S)^{\sigma(0)}\right)^{\frac{1}{h}} \times_{S^h} \prod_{i=1}^{h}  \Sym_{n_0}(T_S/S)^{\sigma(i)}, 
\end{equation}
where the union is taken over cycle decompositions such that $\sigma(i)$ is of size $n_0$. 

 Let  $C(\sigma)$ be the centraliser group of $\sigma$, and $C(\sigma(i))$ be the centraliser group of $\sigma(i)$ viewed as a subgroup of  the symmetric group $\Sigma_{\ell(\tilde{\mu}^i)}$.  We have the following identity
 \begin{multline*} 
\Big[\Big(\coprod_{(\sigma(0),  \hdots , \sigma(h))}  \left(	\widetilde{\Sym}^{\epsilon_+}_{n',[m+h]}(S)^{\sigma(0)}\right)^{\frac{1}{h}} \times_{S^h} \prod_{i=1}^{h}  \Sym_{n_0}(T_S/S)^{\sigma(i)}\Big)/C(\sigma)\Big]\\
 = 
  \coprod_{(\tilde{\mu}^0,  \hdots , \tilde{\mu}^h )}  \Big[ \Big(\left(	\widetilde{\Sym}^{\epsilon_+}_{n',[m+h]}(S)^{\sigma(0)}\right)^{\frac{1}{h}} \times_{S^h} \prod_{i=1}^{h}  \Sym_{n_0}(T_S/S)^{\sigma(i)}\Big)/\prod^h_{i=1}C(\sigma(i))\Big],
 \end{multline*}
where in the second line we take the union over decompositions of the partition into subpartitions $\tilde{\mu}^i$ instead of the decomposition of the permutation into subpermutations, such that $\sigma(i)$ are some representatives of subpartitions $\tilde{\mu}^i$.  
  With respect to this identity,  on (\ref{masterfixed}), the orbifold class $\lambda_{\sS}(\mu)$ is expressed as follows 
\begin{multline*}
\frac{1}{\mathfrak{z}(\mu)}\tau^*\Big(\sum_{\rho\in C(\sigma) } \rho^*( \gamma_{i_1}\otimes \hdots \otimes \gamma_{i_{\ell(\mu)}})\Big)\\
=\sum_{(\mu^0,  \hdots , \mu^h)} \prod^h_{i=1}\frac{1}{\mathfrak{z}(\mu^i)} \tau^* \Big(  \Big(\sum_{\rho \in C(\sigma(0))} 
 \rho^*(\underline{\gamma_1}) \Big) \otimes \hdots  \otimes \Big(\sum_{\rho \in C(\sigma(h))} 
 \rho^*(\underline{\gamma_h}) \Big)\Big)\\ 
  = \sum_{(\mu^0,  \hdots , \mu^h)}  \lambda_{\mathsf{H}}(\mu^0)\otimes \bigotimes^h_{i=1} \lambda_{\mathsf{H}}(\mu^i)(z),
\end{multline*}
where 
\[
\underline{\gamma_\ell}:=\bigotimes_{\mu_j \in \mu^i }\gamma_{i_j},
\] 
and the sum is now taken over weighted decompositions of the partition $\mu$. Substituting (\ref{weighttwist}), we  obtain the statement of the lemma.  

\end{proof}
 
\section{Maps to  master spaces} \label{Mapsmaster}
\subsection{Relative spaces of maps to master spaces} Due to the relative description of moduli spaces of maps to interpolating spaces in Section \ref{Relative}, we can use $\BM \mathcal{FM}_{[m]}(S)$-relative moduli spaces of maps to master spaces to compare the invariants. Let 
\begin{align*} \label{relativemaps}
\Mbar_{g,k}^\rel(\BM\Hilb^{\epsilon_0}_{n,[m]}(S), d) \quad  \text{and} \quad   \Mbar_{g,k}^\rel(\BM\Sym^{\epsilon_0}_{n,[m]}(S), 0)
\end{align*}
be the $\BM\mathcal{FM}_{[m]}(S)$-relative moduli space of maps to mater spaces $\BM\Hilb_{n,[m]}^{\epsilon_0}(S)$ and $\BM\Sym^{\epsilon_0}_{n,[m]}(S)$. By definition,  a $B$-valued point of $\Mbar_{g,k}^\rel(\BM\Hilb^{\epsilon_0}_{n,[m]}(S), d)$ is given by a commutative diagram 
  	\begin{equation} \label{fiberprod1}
	\begin{tikzcd}[row sep = scriptsize, column sep = scriptsize]
		&(\CC,\underline{p})  \arrow[d]  \arrow[r,"f"] &  \BM\Hilb_{n,[m]}^{\epsilon_0}(S) \arrow[d] \\
		& B \arrow[r]&  \BM\mathcal{FM}_{[m]}(S),
	\end{tikzcd}
\end{equation}
such that over closed points we have 
\[ 
f_*([C])=d\iota^!(E). 
\]
A version of Proposition \ref{Proprel} also holds in this case. 
\begin{proposition} \label{masterrelative} The relative moduli spaces of maps 
	\[
	\Mbar_{g,k}^\rel(\BM\Hilb^{\epsilon_0}_{n,[m]}(S), d) \quad  \text{and} \quad   \Mbar_{g,k}^\rel(\BM\Sym^{\epsilon_0}_{n,[m]}(S), 0)
	\]
	are proper.
\end{proposition}	

\begin{proof} The proof is identical to the one of Proposition \ref{Proprel}, whose argument relies on Lemma \ref{exceptionalcurve} and the properness of $\epsilon$-weighted moduli spaces. Lemma \ref{exceptionalcurve}  holds in this case, and the properness of  master spaces is due to Theorem \ref{masterproper}. 
\end{proof}	

We equip $\Mbar_{g,k}^\rel(\BM\Hilb^{\epsilon_0}_{n,[m]}(S), d)$ with the $\BM \mathcal{FM}_{[m]}(S)$-relative obstruction theory given by 
\[
\BT^{\mathrm{vir}}_{\Mbar_{g,k}^\rel(\BM\Hilb^{\epsilon_0}_{n,[m]}(S), d) /\BM\CF\CM_{[m]}(S,n)} :=R\pi_*(F^*(T^{\mathsf{rel}})),
\]
where $T^{\mathsf{rel}}$ is the relative tangent bundle of 
\[ 
 \BM\Hilb^{\epsilon_0}_{n,[m]}(S) \rightarrow \BM \mathcal{FM}_{[m]}(S). 
\]
The same applies to $\Mbar_{g,k}^\rel(\BM\Sym^{\epsilon_0}_{n,[m]}(S), 0)$. These are perfect obstruction theories. 
 
\subsection{Fixed components of spaces of maps to  master spaces} \label{mapsfixed} Consider the $\BC^*$-actions on 
\begin{align*} \label{relativemaps2}
	\Mbar_{g,k}^\rel(\BM\Hilb^{\epsilon_0}_{n,[m]}(S), d) \quad  \text{and} \quad   \Mbar_{g,k}^\rel(\BM\Sym^{\epsilon_0}_{n,[m]}(S), 0),
\end{align*}
induced by the $\BC^*$-actions on the master spaces, or, equivalently, on  the base $\BM\mathcal{FM}_{[m]}(S,n)$.  We will now analyse the associated $\BC^*$-fixed components, focusing on 	$\BM\Hilb^{\epsilon_0}_{n,[m]}(S)$; the same discussion applies to $\BM\Sym^{\epsilon_0}_{n,[m]}(S)$. 

Like in the case of master spaces $\BM\Hilb^{\epsilon_0}_{n,[m]}(S)$, the fixed locus of the moduli spaces of maps takes the following form, 

\begin{equation} \label{fixed}
	\BM \Hilb^{\epsilon_0}_{n,[m]}(X)^{\BC^*}=M_- \sqcup M_+ \sqcup \coprod_{h, \underline{d}} M_{h,\underline{d}}, 
\end{equation} 
where the union is over integers $h\geq 1$ and tuples of non-negative integers $d_i$,
\[
\underline{d}=(d_0,\hdots, d_h) \in \BZ_{\geq 0}^{h+1}. 
\]
The precise description of the components in the above decomposition is given below.

\subsubsection{Fixed component $M_-$}This is the component corresponding to calibrated FM degenerations of the form
\[
[(W,\underline{x}, c)] \in \p(\BM)(\BC) \subset \BM\mathcal{FM}_{[m]}(S,n)(\BC). 
\]
Since the space of maps are $\BM\mathcal{FM}_{[m]}(S,n)$-relative, the component parametrizing tuples with such calibration   admits the following identification,  
\[
M_-= 	\Mbar_{g,k}^\rel(\Hilb^{\epsilon_-}_{n,[m]}(S), d) \quad
N_{M_-}=\bz^{-1}\otimes \BM^\vee.
\]
where $\BM$ is the calibration bundle $\CO(-\CS_1)$ on $\widetilde{\CF\CM}_{[m]}(S,n)$ pulled back to $M_-$, and $\bz$ is the weight $1$ representation of $\BC^*$.  Moreover, the fixed part of the  obstruction theory of $M_-$ agrees with the obstruction theory of $\Mbar_{g,k}^\rel(\Hilb^{\epsilon_-}_{n,[m]}(S), d)$, hence, 
\[
[M_-]^{\mathrm{vir}}=[\Mbar_{g,k}^\rel(\Hilb^{\epsilon_-}_{n,[m]}(S), d)]^{\mathrm{vir}}.
\]
\subsubsection{Fixed component $M_+$}   This is the component corresponding to calibrated FM degenerations of the form 

\[
[(W,\underline{x}, c)] \in \p(\CO)(\BC) \subset \BM\mathcal{FM}_{[m]}(S,n)(\BC). 
\]
Since the space of maps are $\BM\mathcal{FM}_{[m]}(S,n)$-relative, the fixed component admits the following identification,
\[
F_+=\widetilde{M}_{g,k}^\rel(\Hilb^{\epsilon_+}_{n,[m]}(S), d),\quad
N_{F_+}^{\mathrm{vir}}=\mathbf{z}\otimes \BM,
\]
where 
\[
\widetilde{M}_{g,k}^\rel(\Hilb^{\epsilon_+}_{n,[m]}(S), d):=  \Mbar_{g,k}^\rel(\Hilb^{\epsilon_+}_{n,[m]}(S), d) \times_{\CF\CM_{[m]}(S,n)} \widetilde{\CF\CM}_{[m]}(S,n),
\] 
and the other gadgets are as for $M_-$. As in the case of $M_+$, the obstruciton theories also agree,
\[
[M_+]^{\mathrm{vir}}=[\widetilde{M}_{g,k}^\rel(\Hilb^{\epsilon_+}_{n,[m]}(S), d)]^{\mathrm{vir}},
\]
such that the class above is pushed forward to the one of $\Mbar_{g,k}^\rel(\Hilb^{\epsilon_+}_{n,[m]}(S), d)$ with respect to the natural projection. 

\subsubsection{Fixed components $M_{h,\underline{d}}$} 

An element of $\Mbar_{g,k}^\rel(\BM\Hilb^{\epsilon_0}_{n,[m]}(S), d)$ corresponds to a tuple 
\[
(C, W, \Gamma, \underline{x}, c),
\]
such that $(W, \underline{x}, c)$ is a calibrated FM degeneration, and $\Gamma \subset  W\times C$ is a family of $0$-dimensional subschemes over $C$. As in the case of $F_h$, the calibration of elements in $M_{h,\underline{d}}$ is of the from
\[
[(W,\underline{x}, c)] \in \BM\mathcal{FM}_{[m]}(S,n)(\BC)\setminus  (\p(\BM)(\BC)  \cup  \p(\CO)(\BC)).
\]
For a tuple $(C, W, \Gamma, \underline{x}, c)$ to be $\BC^*$-fixed, it has to satisfy a version of Lemma \ref{closed}. 

\begin{lemma} \label{closed2}
	Assume  \[
	[(W,\underline{x}, c)] \in \BM\mathcal{FM}_{[m]}(S,n)(\BC)\setminus  (\p(\BM)(\BC)  \cup  \p(\CO)(\BC)),
	\]
	then a tuple 
	\[
	(C, W, \Gamma, \underline{x}, c) \in 	\Mbar_{g,k}^\rel(\BM\Hilb^{\epsilon_0}_{n,[m]}(S), d)
	\] is $\BC^*$-fixed if and only if (1) $W$ has an end component $P$ of  weight $n_0$, such that $\Gamma |_{P\times C}$ is of length\footnote{We say that an irreducible subscheme $\Gamma$ is of length $n_0$ if its fiber over $C$ is of length $n_0$.} $n_0$ supported on $\{x\}\times C$ and fixed by the $\BC^*$-scaling of the end component; (2) all such end components are entangled. 
\end{lemma}	

\begin{proof} The proof is identical to the one of Lemma \ref{closed}, see also Corollary \ref{threefold}. 
\end{proof}	

By Lemma \ref{closed2}, the types of $\BC^*$-fixed tuples $(C, W, \Gamma, \underline{x}, c)$ are classified by
\begin{itemize}
	\item the number of entangled end components $P_1, \hdots, P_h$;
	\item the Euler characteristics of $\Gamma|_{P_i\times C}$, or, equivalently, the degree $d_i$ of the induced map.   
\end{itemize}	
The fixed components $M_{h,\underline{d}}$ therefore have two indices corresponding to the data above. More explicitly, Lemma \ref{closed2} further implies that each type of  $\BC^*$-fixed tuples $(C, W, \Gamma, \underline{x}, c)$ can be described by the following data:
\begin{itemize}
\item entangled end components $P_i$'s and subschemes  $\Gamma|_{P_i\times C}$;
\item the closure of the complement of $P_i$'s inside $(W, \underline{x})$, denoted by $(W', \underline{x}, \underline{x}')$ (where $\underline{x}'$ are the points of the attachment of $P_i$'s),  and the subscheme $\Gamma|_{W'\times C}$;
\item an entanglement of $W'$ (which canonically lifts to the calibration  of $W$ by the argument from Proposition \ref{important}). 
\end{itemize}

Their  global description is given in the proposition below, where
\[
 \widetilde{\Hilb}^{\epsilon_+}_{n',[m+h]}(S)^{\frac{1}{h}} \times_{S^h} \prod_{i=1}^{h}  \Hilb_{n_0}(T_S/S)
\]
is treated as a  $\widetilde{\CF\CM}_{[m+h]}(S,n')$-relative stack, and the associated space of maps is taken relative to $\widetilde{\CF\CM}_{[m+h]}(S,n')$.

\begin{proposition} \label{important2}  There exists a natural isomorphism, 
	\[
	\widetilde{\mathrm{gl}}^*_hM_{h,\underline{d}}\cong  \Mbar_{g,k}^\rel\Big(\widetilde{\Hilb}^{\epsilon_+}_{n',[m+h]}(S)^{\frac{1}{h}} \times_{S^h} \prod_{i=1}^{h}  \Hilb_{n_0}(T_S/S), \sum^{h}_{i=0} d_i\Big)^{\BC^*},
	\]
	where the $\BC^*$-action on the right is given by the diagonal  scaling of  $h$ copies of the tangent bundle $T_S$; the non-negative integer $d_i$ specifies the degree of maps on the corresponding factor. 
	
\end{proposition}

\begin{proof} The statement can be proved by employing the same arguments as in the proof of Proposition \ref{important}. A shortcut is provided by the morphism from (\ref{smooth}),  
	
\begin{equation*}
	\pi \colon \widetilde{\Hilb}^{\epsilon_+}_{n',[m+h]}(S)^{\frac{1}{h}} \times_{S^h} \prod_{i=1}^{h}  \Hilb_{n_0}(T_S/S)\rightarrow \widetilde{\mathrm{gl}}^*_h \BM \Hilb^{\epsilon_0}_{n,[m]}(S).
\end{equation*}
Since the moduli spaces of maps are $\BM\mathcal{FM}_{[m]}(S,n)$-relative, we can consider their pull-back 
\[
\widetilde{\mathrm{gl}}^*_h\Mbar_{g,k}^\rel(\BM\Hilb^{\epsilon_0}_{n,[m]}(S), d).
\]	
By Lemma \ref{closed2}, the $\BC^*$-fixed locus $\widetilde{\mathrm{gl}}^*_hM_{h,\underline{d}}$ is contained in this pull-back. After removing the translation automorphisms, which corresponds to  pulling back further by the morphism, 
\[
\widetilde{\CF\CM}_{[m+h]}(S,n') \xrightarrow{\alpha} \widetilde{\CF\CM}_{[m+h]}(S,n')\times_{S^h} \CB^h, 
\]
we obtain that the pull-back morphism 
\[
\alpha^*(\pi) \colon \widetilde{\Hilb}^{\epsilon_+}_{n',[m+h]}(S)^{\frac{1}{h}} \times_{S^h} \prod_{i=1}^{h}  \Hilb_{n_0}(T_S/S)  \rightarrow \alpha^*(\widetilde{\mathrm{gl}}^*_h \BM \Hilb^{\epsilon_0}_{n,[m]}(S))
\] 
 is in fact a degree-one \'etale morphism, such that its image is exactly the locus of FM degenerations whose all of the end components $P_i$'s are entangled. Moreover, the $\BC^*$-fixed locus is still naturally contained inside the pull-back, 
 \[
\widetilde{\mathrm{gl}}^*_hM_{h,\underline{d}} \cong  \alpha^*(\widetilde{\mathrm{gl}}^*_hM_{h,\underline{d}})\subseteq  \alpha^*( \widetilde{\mathrm{gl}}^*_h\Mbar_{g,k}^\rel(\BM\Hilb^{\epsilon_0}_{n,[m]}(S), d)),
  \]
as $\alpha$ just removes a moving part of the deformation theory given by the translation automorphism of end components. Finally, since the $\BC^*$-fixed tuples must have entangled end components, we obtain that the images of $\BC^*$-fixed maps must be contained in the image of $\alpha^*(\pi)$. Hence 
\[
\alpha^*(\widetilde{\mathrm{gl}}^*_hM_{h,\underline{d}})\cong  \Mbar_{g,k}^\rel\Big(\widetilde{\Hilb}^{\epsilon_+}_{n',[m+h]}(S)^{\frac{1}{h}} \times_{S^h} \prod_{i=1}^{h}  \Hilb_{n_0}(T_S/S), \sum^{h}_{i=0} d_i\Big)^{\BC^*},
\]
which proves the claim.
\end{proof}

\subsubsection{Normal complexes of fixed components  $M_{h,\underline{d}}$} \label{normalmaps}
To determine the normal complexes of $M_{h,\underline{d}}$, we split the virtual tangent complex 
\[
 \BT^{\mathrm{vir}}_{ 	\Mbar_{g,k}^\rel(\BM\Hilb^{\epsilon_0}_{n,[m]}(S), d) }
 \]
  into the $\BM\CF\CM_{[m]}(S,n)$-relative virtual tangent complex
  \[
 \BT^{\mathsf{rel}}_{\Mbar_{g,k}^\rel(\BM\Hilb^{\epsilon_0}_{n,[m]}(S),d)}:=  \BT^{\mathrm{vir}}_{\Mbar_{g,k}^\rel(\BM\Hilb^{\epsilon_0}_{n,[m]}(S), d) /\BM\CF\CM_{[m]}(S,n)}
  \]
   and the tangent complex of $\BM\CF\CM_{[m]}(S,n)$,
\begin{equation*}
\BT^{\mathsf{rel}}_{\Mbar_{g,k}^\rel(\BM\Hilb^{\epsilon_0}_{n,[m]}(S),d)}
\rightarrow \BT^{\mathrm{vir}}_{ 	\Mbar_{g,k}^\rel(\BM\Hilb^{\epsilon_0}_{n,[m]}(S), d) } \rightarrow \BT_{\BM\CF\CM_{[m]}(S,n)} \rightarrow,
\end{equation*}
where we suppress the pull-back by the projection morphism. For brevity, will we denote
\[
\Mbar^\rel_{g,k}:= \Mbar_{g,k}^\rel\Big(\widetilde{\Hilb}^{\epsilon_+}_{n',[m+h]}(S)^{\frac{1}{h}} \times_{S^h} \prod_{i=1}^{h}  \Hilb_{n_0}(T_S/S), \sum^{h}_{i=0} d_i\Big).
\]
Also, recall the $\psi$-classes associated to the cotangent line bundles $\BL(D_i)$ and $\BL(E_i)$, 
\[
\psi(D_i):=\mathrm{c}_1(\BL(D_i)), \quad \Psi_i=\mathrm{c}_1(\BL(E_i)),
\]
such that the class $\Psi_i$ is the same as one considered in Section \ref{sectioninv}; see \cite{NHilb} for more details about these classes. 
\begin{proposition} \label{isomophism} With respect to the identification from Proposition \ref{important2},  the virtual fundamental class of $M_{h,\underline{d}}$ is described as follows.
\begin{equation*} [\widetilde{\mathrm{gl}}_h^*M_{h,\underline{d}}]^{\mathrm{vir}}	= [(\Mbar^\rel_{g,k})^{\BC^*}]^{\mathrm{vir}},
\end{equation*}
while the Euler class of the virtual normal complex is given by
\begin{multline*}		
		  	\frac{1}{\mathrm{e}_{\BC^*}( N_{\widetilde{\mathrm{gl}}_h^*M_{h,\underline{d}}}^{\mathrm{vir}})}
		  	=\frac{\prod^h_{i=1}(\sum_j \mathrm{c}_j(S)(z/h+\psi(D_i))^{2-j}) \cdot \CI_{h,d_i}(z/h+\psi(D_i)) }{-z/h-\psi(D_1)-\Psi_{m+1}-\sum^{\infty}_{i=h+1}\CY_i}, 
\end{multline*}
where
\[
\prod^h_{i=1}\CI_{h,d_i}(z)=\frac{1}{\mathrm{e}_{\BC^*}\Big(N^{\mathrm{vir}}_{(\Mbar^\rel_{g,k})^{\BC^*}/\Mbar^\rel_{g,k}}\Big)} \in H^*((\Mbar^\rel_{g,k})^{\BC^*})[z^\pm],
\]
and $\CI_{h,d_i}(z)$ are the factors associated to each   $\Hilb_{n_0}(T_S/S)$. 
The same holds for the wall-crossing components of 	$\Mbar_{g,k}^\rel(\BM\Sym^{\epsilon_0}_{n,[m]}(S), 0)$. 
\end{proposition}

\begin{proof} 
	The moving part of $\BT_{\BM\CF\CM_{[m]}(S,n)}$ is equal to the normal bundle of the locus of   $h$ etangled components
	\[
	 \CY^*_h \subset \BM\CF\CM_{[m]}(S,n),
	 \] 
	 and the translation automorphisms of these end components.
	By \cite[Lemma 2.5.5]{YZ}, the Chern classes of the former is equal to 
	 \[
	 -z/h-\psi(D_1)-\Psi_{m+1}-\sum^{\infty}_{i=h+1}\CY_i.
	 \]
	On the other hand, with respect to the identification \ref{important2}, the translation automorphisms are given by $T_S\otimes \bz$ on each factor, whose Euler class is therefore equal to
	\[
	\sum_j \mathrm{c}_j(S)z^{2-j},
	\]
	which, after the substitution (\ref{weighttwist}), becomes
	\[
	\sum_j \mathrm{c}_j(S)(z/h+\psi(D_i))^{2-j}).
	\]
	The fixed part of $\BT_{\BM\CF\CM_{[m]}(S,n)}$ is given by $\BT_{\BM\CF\CM_{[m+h]}(S,n')}$. 
	
	Next, by the proof of Proposition \ref{important2}, the moving part of  
	\[
	\BT^{\mathsf{rel}}_{\Mbar_{g,k}^\rel(\BM\Hilb^{\epsilon_0}_{n,[m]}(S),d)}
	\]
	 is given by the moving part of 
	\[
	\BT^{\mathsf{rel}}_{\Mbar^\rel_{g,k}}.
	\]
Since the tangent bundle of the  target splits as the direct sum of tangent bundles of its components, the Euler class of the moving part splits a product of the Euler classes of the moving parts associated to each of the factors, giving rise to 
	\[
	\prod^{h}_{i=1} \CI_{h,d_i}(z/h+\psi(D_i))
	\]
	after the substitution of $(\ref{weighttwist})$. The fixed part of the relative tangent complex also coincides with the one of $\BT^{\mathsf{rel}}_{\Mbar^\rel_{g,k}}$. Putting all together, we obtain the claim. 
	\end{proof}

\subsection{Product formula}

By using the product formula in Gromov--Witten theory \cite{Behr}, we can express the virtual fundamental classes and Euler classes of virtual normal complexes of $M_{h,\underline{d}}$ in terms of the individual factors of 
\begin{equation} \label{producthilb}
\widetilde{\Hilb}^{\epsilon_+}_{n',[m+h]}(S)^{\frac{1}{h}} \times_{S^h} \prod_{i=1}^{h}  \Hilb_{n_0}(T_S/S).
\end{equation}

First, note that 
\[
\Mbar^\rel_{g,k}(\widetilde{\Hilb}^{\epsilon_+}_{n',[m+h]}(S)^{\frac{1}{h}} ,d_0)\cong \Mbar^\rel_{g,k}(\widetilde{\Hilb}^{\epsilon_+}_{n',[m+h]}(S),d_0)^{\frac{1}{h}}\]
because the moduli spaces are $\widetilde{\CF\CM}_{[m+h]}(S,n')$-relative and the $h$-root is associated to the line-bundle pulled back from $\widetilde{\CF\CM}_{[m+h]}(S,n')$. Next, there is a natural projection
\[
\pi \colon \Mbar^\rel_{g,k} \rightarrow \Mbar^\rel_{g,k}(\Hilb^{\epsilon_+}_{n',[m+h]}(S),d_0).
\]
We define the pushforward of the $\BC^*$-equivariant virtual fundamental capped with an insertion $A$ from the target, 
\[ 
C(z):=\pi_*\bigg(\frac{[(\Mbar^\rel_{g,k})^\mathrm{\BC^*}]^{\mathrm{vir}}}{\mathrm{e}_{\BC^*}(N^\mathrm{vir})} \cap A \bigg) \in H_*(\Mbar^\rel_{g,k}(\Hilb^{\epsilon_+}_{n',[m+h]}(S),d_0))[z^\pm],
\]
such that 
\[
A=A_0 \cdot A_1 \cdot \hdots A_h,
\]
where $A_i$ comes from the factors of the product (\ref{producthilb}) via the evaluation morphisms.
There also exist the following morphisms
\[
\widetilde{\ev}_i \colon \Mbar^\rel_{g,k}(\Hilb^{\epsilon_+}_{n',[m+h]}(S),d_0) \rightarrow  S \times \Mbar_{g,k}, \quad i=1, \hdots, h,
\]
such that over the first factor they are given by the evaluation morphisms associated to the points $x_{m+i}$ on FM degenerations (see Section \ref{sectioninv}). 

While for $\Hilb_{n_0}(T_S/S)$, we also have a projection from the  $S$-relative space of maps, 
\[
\tau \colon \Mbar^\rel_{g,k}(  \Hilb_{n_0}(T_S/S),d_i) \rightarrow S \times \Mbar_{g,k}.
\]
We define the pushforward  of $\BC^*$-equivariant virtual fundamental  classes capped with the insertion $A_i$ from $\Hilb_{n_0}(T_S/S)$, 
\[
I_i(z):=\tau_*\bigg(\frac{[\Mbar^\rel_{g,k}(  \Hilb_{n_0}(T_S/S),d_i)^\mathrm{\BC^*}]^{\mathrm{vir}}}{\mathrm{e}_{\BC^*}(N^\mathrm{vir})} \cap A_i\bigg)  \in H^*(S \times \Mbar_{g,k})[z^\pm]. 
\]
The next result  is known as a \textit{product formula}. 
\begin{proposition}[ \hspace{-0.23cm} \cite{Behr}] \label{prodformula}
	We have the following identity in the homology, 
	\[
	C(z)=\frac{1}{h} ([\Mbar^\rel_{g,k}(\Hilb^{\epsilon_+}_{n',[m+h]}(S),d_0)]^{\mathrm{vir}} \cap A_0) \cap (\widetilde{\ev}_1^*I_1(z) \cdot  \hdots \cdot  \widetilde{\ev}_h^*I_h(z)).
	\]
\end{proposition}

\begin{proof} This follows from the relative version of main result of \cite{Behr}. The prefector is due to the $h$-roots. There is no sign involved because the fiber of $\Hilb_{n_0}(T_S/S)$ over $S$ has only algebraic classes. More explicitly, a relative Nakajima--Grojnowski class $\lambda_{\sH}(\mu)$ on $\Hilb_{n_0}(T_S/S)$, associated to a partition $\mu$ weighted by the cohomology of $S$, decomposes
	\[
	\lambda_{\sH}(\mu)=\prod_j  \gamma_{i_j} \cdot \lambda_{\sH}(\mu_1,\hdots,\mu_{\ell(\mu)}),
	\]    
	where $\lambda_{\sH}(\mu_1,\hdots,\mu_{\ell(\mu)})$ is the class associated to the unweighted partition. 
	
	Hence, all odd classes come from the base $S$, and therefore can equivalently be inserted by the projection morphism to $S^h$. 
	 The latter  commutes with the projection $\pi$, allowing us to move these insertions freely among  algebraic  insertions associated to $\Hilb_{n_0}(T_S/S)$  without affecting the sign.  
\end{proof}	

The extension of the result to the case involving insertions coming from  $\widetilde{\CF\CM}_{[m+h]}(S,n')\times \Mbar_{g,k}$ is obvious: we simply cap both sides by that insertion. Also, the same product formula  holds for $\Mbar_{g,k}^\rel(\BM\Sym^{\epsilon_0}_{n,[m]}(S), 0)$. 

\section{Wall-crossing} 
\subsection{$I$-functions} \label{secifunction} Below, we introduce classes in  $H^*(S\times \Mbar_{g,k})[z^\pm]$, called  $I$-functions. They are responsible for the wall-crossing of Gromov--Witten invariants of Hilbert schemes and Symmetric products, and are associated to the wall-crossing components, which are described in Proposition \ref{isomophism}. We define them via  $S$-relative moduli spaces of maps to the total space of the tangent bundle $T_S$, 
\begin{equation*}
I^{\mathsf{H}_n}_{g,d}\Bigg(z, \prod^{k}_{j=1} \lambda_{\mathsf{H}}(\mu^j)\Bigg):= 
\mathrm{e}_{\BC^*}( T_S  \otimes \bz ) \cdot \tau_* \left(  \frac{ \prod^k_{j=1} \ev^*_j\lambda_{\mathsf{H}}(\mu^j)  \cap [V^S_{n,d}]^{\mathrm{vir}}}{\mathrm{e}_{\BC^*}(N_{V^S_{n,d}}^\mathrm{vir})}\right) 
\end{equation*}
where
\[ 
V^S_{n,d}:= \Mbar_{g,k}^\rel( \Hilb_{n}(T_S/S),d )^{\BC^*},
\]
and 
\[ 
\tau \colon \Mbar_{g,k}^\rel( \Hilb_n(T_S/S), d) \rightarrow S\times \Mbar_{g,k} 
\]
is the natural projection associated to the $S$-relative moduli spaces of maps to $T_S \rightarrow S$. 

We have the same definition for $I$-functions responsible for the wall-crossing of the  symmetric product,
\begin{equation*}
	I^{\mathsf{S}_n}_{g}\Bigg(z, \prod^{k}_{j=1} \lambda_{\mathsf{S}}(\mu^j)\Bigg):= 
	\mathrm{e}_{\BC^*}(T_S \otimes \bz  ) \cdot \tau_* \left(  \frac{ \prod^k_{j=1} \ev^*_j\lambda_{\mathsf{S}}(\mu^j)  \cap [V^S_{n}]^{\mathrm{vir}}}{\mathrm{e}_{\BC^*}(N_{V^S_{n}}^\mathrm{vir})}\right) 
\end{equation*}
where
\[ 
V^S_{n}:= \Mbar_{g,k}^\rel( \Sym_{n}(T_S/S),0)^{\BC^*},
\]
and  $\tau$ is the same kind of projection. Note that we Taylor expand  the rational functions in $z$ in the range $|z| >1$; this also applies to the substitution of variables in the next result. 
\begin{lemma} \label{Iuniversal} Let $T$ be the 2-torus acting on $\BC^2$ with weights $t_1$ and $t_2$, and let $\alpha_1(S)$ and $\alpha_2(S)$ be the Chern roots of $S$. Under the change of variables\footnote{Note that this change of variables is well-defined because the Gromov--Witten invariants of $\Hilb_n(\BC^2)$ are symmetric Laurent polynomials in variables $t_1$ and $t_2$; this follows from the properness of the fixed loci associated to other weights of $T$ and the symmetry of $x$- and $y$-coordinates of $\BC^2$.   }
	\[
	t_i \mapsto z+\alpha_i(S),
	\] 
	we have have the following relation in $H^*(S\times \Mbar_{g,k})[z^\pm]$, 
	\begin{multline*}
	I^{\mathsf{H}_n}_{g,d}\Bigg(z, \prod^{k}_{j=1} \lambda_{\mathsf{H}}(\mu^j)\Bigg)
	=   \mathrm{e}_{\BC^*}( T_S \otimes \bz   )  \cdot \prod^{\ell(\mu)}_{j=1} \gamma_{i_j}  \cdot \tau^{\BC^2}_*\left(  \frac{ \prod^k_{j=1} \ev^*_j\lambda_{\mathsf{H}}(\mu^j)  \cap [V^{\BC^2}_{n,d}]^{\mathrm{vir}}}{\mathrm{e}_{T}(N_{V^{\BC^2}_{n,d}}^\mathrm{vir})}\right) ,
	\end{multline*}
where
\[
\tau^{\BC^2} \colon V^{\BC^2}_{n,d}:= \Mbar_{g,k}( \Hilb_{n}(\BC^2),d )^{T} \rightarrow \Mbar_{g,k},
\]
and, on $\BC^2$, the class $\lambda_{\mathsf{H}}(\mu)$ associated to a partition $\mu$ cohomologically weighted by classes on $S$ is defined to be one associated to the underlying partition  $(\mu_1, \hdots, \mu_{\ell(\mu)})$.

	The same holds for the  symmetric product, 
	\begin{multline*}
	I^{\mathsf{S}_n}_{g}\Bigg(z, \prod^{k}_{j=1} \lambda_{\mathsf{S}}(\mu^j)\Bigg) 
	= \mathrm{e}_{\BC^*}( T_S \otimes \bz   )  \cdot \prod^{\ell(\mu)}_{j=1} \gamma_{i_j}  \cdot \tau^{\BC^2}_*\left(  \frac{ \prod^k_{j=1} \ev^*_j\lambda_{\mathsf{S}}(\mu^j)  \cap [V^{\BC^2}_{n}]^{\mathrm{vir}}}{\mathrm{e}_{T}(N_{V^{\BC^2}_{n}}^\mathrm{vir})}\right),
	\end{multline*}
where
\[
\tau^{\BC^2} \colon V^{\BC^2}_{n,d}:= \Mbar_{g,k}( \Sym_{n}(\BC^2),0)^{T} \rightarrow \Mbar_{g,k},
\]
and, on $\BC^2$,  the class $\lambda_{\mathsf{S}}(\mu)$ is also associated to the underlying unweighted partition. 
	\end{lemma}

\begin{proof} The same argument as in \cite[Proposition 4.1]{NHilb} applies. The proof uses the $S$-relative flag variety of $T_S$ and a deformation argument to split the pullback of $T_S$ into a sum of line bundles. For a split vector bundle, the claim follows by taking a base change to the universal trivialisation of the vector bundle. The passage between all these spaces is realised via the base change of $S$-relative maps. Hence,  the virtual fundamental classes are preserved with respect to  pull-backs. 
\end{proof}	

\subsection{Wall-crossing formulas}

\begin{theorem} \label{wsH}We have 
	\begin{multline*}
	\langle \lambda_{\mathsf H}(\mu^1), \dots, \lambda_{\mathsf H}(\mu^k) \mid \Psi_1^{h_1}\gamma_1, \hdots, \Psi_m^{h_m}\gamma_m \rangle^{\sH_{n},\epsilon_-}_{g,d} \\
	=\sum_{h, \underline{d}, \underline{\mu} } \frac{1}{h!}\Bigg\langle \lambda_{\sH}(\mu^{1,0}), \hdots,\lambda_{\sH}(\mu^{k,0}) \ \Bigg|  \	\Psi_1^{h_1}\gamma_1, \hdots, \Psi_m^{h_m}\gamma_m, \\
  \prod^{h}_{i=1}I^{\mathsf{H}_{n_0}}_{g,d_i}\Bigg(-\Psi_{m+i}, \prod^{k}_{j=1} \lambda_{\mathsf{H}}(\mu^{j,i})\Bigg) \Bigg\rangle^{\sH_{n-hn_0},\epsilon_+}_{g,d_0},  
	\end{multline*}
	where we sum over integers $h\geq 0$, length $h+1$ partitions of $d$, length $h+1$ decompositions of cohomologically weighted partitions $\mu^j$, 
	\[ 
\mu^j=\mu^{j,0} \sqcup \mu^{j,1} \sqcup \hdots \sqcup \mu^{j,h}, 
	\]
	such that $|\mu^{j,0}|=n-hn_0$, $\mu^{j,i}=n_0$ for $i\geq 1$, and $d_i\geq 0$ for $i\geq 0$. Negative powers of $\psi$-classes are set to be zero. 
	\end{theorem}
\begin{proof}
	
	Given a wall $\epsilon_0 \in \BR_{>0}$ separating $\epsilon_-$ and $\epsilon_+$. By Proposition \ref{masterrelative},  there exists a proper relative  moduli space of maps to the master space,
	\[
		\Mbar_{g,k}^{\mathsf{rel}}(\BM\Hilb^{\epsilon_0}_{n,[m]}(S), d).
	\]
	Let be the insertions associated to markings on curves 
	\[
	A= \prod^{k}_{j=1}\ev_j^*\lambda_{\mathsf H}(\mu^j)
	\]
	and 
	\[
	B=\prod^{m}_{i=1}\Psi_i^{h_i} \widetilde{\ev}_i^*\gamma_i 
	\]
	be those associated to markings on FM degenerations. The $\BC^*$-fixed components of the moduli spaces of maps are described in Section \ref{mapsfixed}.  Since the space is proper, the $z=0$ residue of the localisation formula applied to 
	\[
	\int_{[\Mbar_{g,k}^{\mathsf{rel}}(\BM\Hilb^{\epsilon_0}_{n,[m]}(S), d)]^{\mathrm{vir}}} A \cdot B \in \BQ[z]
	 \]
	 gives rise to the following relation between the integrals on the $\BC^*$-fixed components, 
	 
	 		\begin{multline}\label{residue0}
	\int_{[\Mbar_{g,k}^{\mathsf{rel}}(\Hilb^{\epsilon_+}_{n,[m]}(S), d)]^{\mathrm{vir}}} A\cdot B - 	\int_{[\Mbar_{g,k}^{\mathsf{rel}}(\Hilb^{\epsilon_-}_{n,[m]}(S), d)]^{\mathrm{vir}}}
	 A\cdot B \\
	 + \mathrm{Res}_{z} \sum_{h, \underline{d}} \int_{[M_{h,\underline{d}}]^{\mathrm{vir}}}A\cdot B =0,
	\end{multline}
where we used the descriptions of $M_+$ and $M_-$ from Section \ref{mapsfixed}, and the fact that their normal bundles are line bundles (hence, the residues are equal to $\pm1$). 

After making  the change of variables, which affects the residue by the factor of $h$,  
\[ 
z \mapsto h(z- \Psi(D_1)-\Psi_{m+1})=\hdots=h(z- \Psi(D_h)-\Psi_{m+h}),
\]
and using Propositions \ref{important2} and \ref{isomophism}, we obtain that  
\begin{multline*}
\mathrm{Res}_{z} \int_{[M_{h,\underline{d}}]^{\mathrm{vir}}}A\cdot B=\frac{h}{h!} \mathrm{Res}_z\int_{[(\Mbar_{g,k}^{\mathsf{rel}})^{\BC^*}]^{\mathrm{vir}}}\frac{A\cdot B}{-z-\sum^\infty_{i=h+1}{\CY_i}} \\
\cdot \prod^{h}_{i=1}(\sum_j \mathrm{c}_j(S)(z-\Psi_{m+i})^{2-j})
\cdot \CI_{h,d_i}(z-\Psi_{m+i}),
\end{multline*}
where the prefactor $h/h!$ is due to the change of variables and the gluing morphism $\widetilde{\mathrm{gl}}_h$. Next, using the splitting formula for Nakajima--Grojnwoski classes, Lemma \ref{Nsplitting}, and the product formula for Gromov--Witten classes, Proposition \ref{prodformula}, we obtain that the above residue is equal to 
\begin{multline*}
\mathrm{Res}_z \sum_{\underline{\mu}}\frac{1}{h!} \int_{[\Mbar^{\mathsf{rel}}_{g,k}(\Hilb^{\epsilon_+}_{n',[m+h]}(S),d_0)]^{\mathrm{vir}}} \frac{A_0  \cdot B}{-z-\sum^\infty_{i=h+1}\CY_i}  \\
\cdot \prod^{h}_{i=1} \widetilde{\ev}_{m+i}^*I^{\mathsf{H}_{n_0}}_{g,d_i} (z-\Psi_{m+i},A_i ),
\end{multline*}
where 
\[
A_i= \prod^{k}_{j=1}\ev_j^*\lambda_{\mathsf H}(\mu^{j,i}), \quad i=0,\hdots, h
\]
are the insertions corresponding to the splitting ot Nakajima--Grojnowski classes from Lemma \ref{Nsplitting}. 

It remains to analyse the contributions of $\CY_i$ to the residue above. Recall that these are effective divisors whose generic point parametrises entangled FM degenerations with $i$ entangled end components. The restrictions of  $\Mbar^{\mathsf{rel}}_{g,k}(\Hilb^{\epsilon_+}_{n',[m+h]}(S),d_0)$ to these divisors further split the moduli spaces into the products of the form as in Proposition \ref{important2}. Below, we will only state the result of their contribution, referring to the proofs of  \cite[Theorem 7.3.3]{YZ} and \cite[Theoreom 8.9]{NuGW} for further details, as the analysis is identical and depends mainly on the geometry of calibrated degenerations. Hence, after incorporating the contributions of $\CY_i$, we obtain that  
\begin{multline}
\mathrm{Res}_{z} \sum_{h, \underline{d}} \int_{[M_{h,\underline{d}}]^{\mathrm{vir}}}A\cdot B \\
=\sum_{h, \underline{d}, \underline{\mu}}\frac{1}{h!} \int_{[\Mbar^{\mathsf{rel}}_{g,k}(\Hilb^{\epsilon_+}_{n',[m+h]}(S),d_0)]^{\mathrm{vir}}} A_0  \cdot B
\cdot \prod^{h}_{i=1} \widetilde{\ev}_{m+i}^*I^{\mathsf{H}_{n_0}}_{g,d_i} (-\Psi_{m+i},A_i ), 
\end{multline}
where the negative powers of $\psi$-classes are set to be zero. This is the desired formula claimed in the theorem; observe that the $\epsilon^+$-term on the left-hand side of (\ref{residue0}) is treated as the $h=0$ term in the statement of the theorem. 
\end{proof}	
The same result holds for invariants associated to the symmetric products. 

\begin{theorem} \label{wsS} We have 
	\begin{multline*}
\langle \lambda_{\mathsf S}(\mu^1), \dots, \lambda_{\mathsf S}(\mu^k) \mid \Psi_1^{h_1}\gamma_1, \hdots, \Psi_m^{h_m}\gamma_m \rangle^{\sS_{n},\epsilon_-}_{g} \\
=\sum_{h, \underline{\mu} } \frac{1}{h!}\Bigg\langle \lambda_{\sS}(\mu^{1,0}), \hdots,\lambda_{\sS}(\mu^{k,0}) \ \Bigg|  \	\Psi_1^{h_1}\gamma_1, \hdots, \Psi_m^{h_m}\gamma_m, \\
\prod^{h}_{i=1}I^{\mathsf{S}_{n_0}}_{g}\Bigg(-\Psi_{m+i}, \prod^{k}_{j=1} \lambda_{\mathsf{S}}(\mu^{j,i})\Bigg) \Bigg\rangle^{\sS_{n-hn_0},\epsilon_+}_{g},  
	\end{multline*}
where the the same conventions as in Theorem \ref{wsH} apply.
\end{theorem}
\begin{proof} Same as for Theorem \ref{wsH}. 
	\end{proof}

By applying Theorem \ref{wsH} and \ref{wsS} inductively, we obtain the following expressions of Gromov--Witten invariants of $\Hilb_n(S)$ and $\Sym_n(S)$ in terms of Fulton--MacPherson integrals and Gromov--Witten invariants of $\Hilb_n(\BC^2)$ and $\Sym_n(\BC^2)$. Observe that for $\epsilon >1$, the spaces $\Hilb^{\epsilon}_{n,[m]}(S)$ and $\Sym^{\epsilon}_{n,[m]}(S)$ are empty, hence, after crossing all walls, the wall-crossing contributions are insertions on  
\[
\Mbar_{g,k}^\rel(\FM_{[h]}(S),0),
\]
which is isomorphic to 
\[
\FM_{[h]}(S) \times \Mbar_{g,k},
\]
such that the virtual fundamental class is equal to the standard class; this leads to the invariants defined in Section \ref{FMspaces}. 
\begin{theorem} \label{cor} In the notation of Sections \ref{FMspaces} and \ref{secifunction}, we have 
\begin{align*}
	&	\langle \lambda_{\mathsf H}(\mu^1), \dots, \lambda_{\mathsf H}(\mu^k) \rangle^{\sH_n}_{g,d}= \sum_{\underline{n}, \underline{d}, \underline{\mu} } \frac{1}{h!}\Bigg\langle  \prod^{h}_{i=1}I^{\mathsf{H}_{n_i}}_{g,d_i}\Bigg(-\Psi_i, \prod^{k}_{j=1} \lambda_{\mathsf{H}}(\mu^{j,i})\Bigg) \Bigg\rangle^{\mathsf{FM}_h}_{g,k} \\
		& \langle \lambda_{\sS}(\mu^1), \dots, \lambda_{\sS}(\mu^k) \rangle^{\sS_n}_{g}= \sum_{\underline{n},\underline{\mu}} \frac{1}{h!} \Bigg \langle  \prod^{h}_{i=1}I^{\sS_{n_i}}_{g}\Bigg(-\Psi_i, \prod^{k}_{j=1} \lambda_{\sS}(\mu^{j,i})\Bigg) \Bigg\rangle^{\mathsf{FM}_h}_{g,k}.
\end{align*}
where we sum over all partitions of $n$,
\[
(n_1, \hdots, n_h),
\]
partitions of $d$, 
\[ 
(d_1, \hdots, d_h),
\]
and decompositions of $\mu^j$ into subpartitions 
	\[ 
\mu^j=\mu^{j,1} \sqcup \hdots \sqcup \mu^{j,h},
\]
such that $n_i>0$,  $d_i\geq 0$, and $|\mu^{j,i}|=n_i$.  Negative powers of $\psi$-classes are set to be zero. 
	\end{theorem}

Let us assemble $I$-functions of Hilbert schemes into a generating series, 
\[
I^{\mathsf{H}_n}_{g}\Bigg(z,q \prod^{k}_{j=1} \lambda_{\mathsf{H}}(\mu^{j})\Bigg):=\sum_{d\geq 0} I^{\mathsf{H}_n}_{g,d}\Bigg(z, \prod^{k}_{j=1} \lambda_{\mathsf{H}}(\mu^{j})\Bigg)q^d \in H^{*}(S)[z^\pm][\![q]\!].
\]
This allows to write the above corollary for Gromov--Witten invariants of Hilbert schemes in the following form. 
\begin{corollary} \label{cor2}We have 
\[ 
\langle \lambda_{\mathsf H}(\mu^1), \dots, \lambda_{\mathsf H}(\mu^m) \rangle^{\sH_n}_{g}(q)=\sum_{\underline{n},\underline{\mu}} \Bigg\langle  \prod^{h}_{i=1}I^{\mathsf{H}_{n_i}}_{g,d_i}\Bigg(-\Psi_i,q,\prod^{m}_{j=1} \lambda_{\mathsf{H}}(\mu^{ij})\Bigg) \Bigg\rangle^{\mathsf{FM}_h}_{g,k},
\]
where we sum over partitions of the integer $n$ and decompositions of partitions $\mu^j$ as in Theorem \ref{wsH}. 
\end{corollary}

\section{Proof of Theorems \ref{Ruan} and  \ref{mainresult}}
\subsection{Proof}
For $2g-2+k>0$, using the notation from Lemma \ref{Iuniversal}, we define $T$-equivariant Gromov--Witten classes of $\Hilb_n(\BC^2)$,
\[
\mathsf{GW}(\mu^1, \hdots, \mu^k)^n_{g,d}=\tau^{\BC^2}_*\left(  \frac{ \prod^k_{j=1} \ev^*_j\lambda_{\mathsf{H}}(\mu^j)  \cap [V^{\BC^2}_{n,d}]^{\mathrm{vir}}}{\mathrm{e}_{T}(N_{V^{\BC^2}_{n,d}}^\mathrm{vir})}\right) \in H^*(\Mbar_{g,k})[t_1^\pm,t_2^\pm], 
\]
and 
\[ 
\mathsf{GW}(\mu^1, \hdots, \mu^k)^n_g(q):=\sum_{d\geq 0}\mathrm{GW}_{g,d}(\mu^1, \hdots, \mu^k)q^d \in H^*(\Mbar_{g,k})[t_1^\pm,t_2^\pm][\![q]\!].
\]
 By \cite[Theorems 2 and 5]{PanT}, the formal series 
 \[
\mathsf{GW}(\mu^1, \hdots, \mu^k)^n_g(q)
 \]
  are the Tailor expansions of rational functions in $q$ with no pole at $q=-1$, i.e., 
  \begin{equation} \label{rationality}
  \mathsf{GW}(\mu^1, \hdots, \mu^k)^n_g(q) \in H^*(\Mbar_{g,k})(t_1,t_2,q).
  \end{equation}
For $g=0$, this was established in \cite[Theorem 1]{OkPa}. In fact, by a direct argument presented in Section \ref{rat}, a stronger rationality property holds: they are contained in the subring
  \begin{equation} \label{strongerrat}
  H^*(\Mbar_{g,k})[t^\pm_1,t^\pm_2] \otimes_{\BQ} \BQ(q) \subset H^*(\Mbar_{g,k})(t_1,t_2,q).
  \end{equation}
  By Lemma \ref{Iuniversal}, the $I$-functions of $\Hilb_n(S)$ can be expressed in terms of $\Hilb_n(\BC^2)$, such that the torus weighs $t_i$ are substituted by $z+\alpha_i(S)$ and then expanded in the range $|z|>1$. Altogether this implies that $I$-functions are rational functions in $q$ with coefficients in $H^*(S)[z^\pm]$ and with no pole at $q=-1$, 
   \[
   I^{\mathsf{H}_n}_{g}\Bigg(z,q,\prod^{k}_{j=1} \lambda_{\mathsf{H}}(\mu^{j})\Bigg) \in H^*(S)[z^\pm](q).
   \]
    In Corollary \ref{cor2}, we substitute $z$ by $\psi$-classes, from which  obtain that 
  \[
 \langle \lambda_{\mathsf H}(\mu^1), \dots, \lambda_{\sH}(\mu^k) \rangle^{\sH_n}_{g}(q) 
 \]
are the Tailor expansions of rational functions with no pole at $q=-1$ for all projective surfaces $S$. 
	Moreover,  by \cite[Theorem 2]{PanT} and Lemma \ref{Iuniversal}, after substituting $q=-1$, we have
	\[ 
	I^{\mathsf{H}_n}_{g}\Bigg(z,q,\prod^{k}_{j=1} \lambda_{\mathsf{H}}(\mu^{j})\Bigg)= 	I^{\mathsf{S}_n}_{g}\Bigg(z,\prod^{k}_{j=1} \lambda_{\mathsf{S}}(\mu^j)\Bigg).
	\]
Hence, by Theorem \ref{cor} and Corollary \ref{cor2}, after substituting $q=-1$, we have
\[
\langle \lambda_1, \dots, \lambda_k \rangle^{\sH_n}_{g}(q)=\langle \lambda_1, \dots, \lambda_k \rangle^{\sS_n}_{g},
\] 
where we use that the shapes of the wall-crossing formulas in Theorem \ref{cor} are identical for both $\Hilb_n(S)$ and $\Sym_n(S)$. 
This finishes the proof of Theorem \ref{mainresult}. Theorem \ref{Ruan} follows from  the identification (\ref{isoL}) and Theorem \ref{mainresult} specialised to $g=0$ and $k=3$.

\subsection{Rationality of $I$-functions} \label{rat}
In section, we will establish that
\[
  \mathsf{GW}(\mu^1, \hdots, \mu^k)^n_g(q) \in H^*(\Mbar_{g,k})[t^\pm_1,t^\pm_2] \otimes_{\BQ} \BQ(q).
\]
First, consider the following proper morphism
\[
\pi \colon \Hilb_n(\BC^2) \rightarrow \BC^n/\Sigma_n \times  \BC^n/\Sigma_n 
\]
which is the composition of the Hilbert--Chow morphism with 
 \[
 (\BC^2)^n/\Sigma_n \rightarrow (\BC^n/\Sigma_n)^2,
 \]
 given by sending a configuration of points to the values of its first and second coordinates. The space $(\BC^n/\Sigma_n)^2$ is smooth. For any class $A\in H^*_T(\Hilb_n(\BC^2))$, we have the following identity 
 \[ 
 \int_{\Hilb_n(\BC^2)^T}\frac{A|_{\Hilb_n(\BC^2)^T}}{\mathrm{e}_T(N_{\Hilb_n(\BC^2)^T})} =
 \int_{((\BC^n/\Sigma_n)^2)^T}\frac{(\pi_*A)|_{((\BC^n/\Sigma_n)^2)^T}}{\mathrm{e}_T(N_{((\BC^n/\Sigma_n)^2)^T})}.
 \] 
 Since $\mathrm{e}_T(N_{((\BC^n/\Sigma_n)^2)^T})=(t_1t_2)^n$, and $((\BC^n/\Sigma_n)^2)^T=\mathrm{pt}$, we obtain that  
 \begin{equation} \label{property}
 \int_{\Hilb_n(\BC^2)^T}\frac{A|_{\Hilb_n(\BC^2)^T}}{\mathrm{e}_T(N_{\Hilb_n(\BC^2)^T})}=\frac{p(t_1,t_2)}{(t_1t_2)^n}
 \end{equation}
 for some polynomial $p(t_1,t_2)$. 
 
 Next, assuming $k>0$, the evaluation morphism 
 \[
 \ev_{1} \colon \Mbar_{g,k}(\Hilb_n(\BC^2), d) \rightarrow \Hilb_n(\BC^2)
 \]
 is proper. Indeed, this can be seen by composing it with the Hilbert--Chow morphism $\Hilb_n(\BC^2) \rightarrow (\BC^2)^n/\Sigma_n$, which is proper, and using the fact that the evaluation morphism from the space of degree 0 maps is proper. Factoring the projection morphism $\tau^{\BC^2}$, 
 \begin{equation*}
 	\begin{tikzcd}
 		& \Mbar_{g,k}(\Hilb_n(\BC^2), d)   \arrow[r,"\ev_1\times \tau^{\BC^2} "] \arrow[d, "\tau^{\BC^2}"] &  \Hilb_{n}(\BC^2)\times \Mbar_{g,k}\arrow[dl] \\
 		& \Mbar_{g,k}, &  
 	\end{tikzcd}
 \end{equation*}
 and using the property (\ref{property}), we obtain that 
 
 \[
\mathsf{GW}(\mu^1, \hdots, \mu^k)^n_{g,d} \in \frac{1}{(t_1t_2)^n}H^*(\Mbar_{g,k})[t_1,t_2]_{m} \subset H^*(\Mbar_{g,k})[t^\pm_1,t^\pm_2], 
 \]
where $H^*(\Mbar_{g,k})[t_1,t_2]_{m}$ is the subspace of homogenous elements of degree $m$. Note that $m$ is constant for all $d$, since the virtual dimension is independent of $d$; the prefactor is also independent by (\ref{property}). Hence, 
 \[
 \mathsf{GW}(\mu^1, \hdots, \mu^k)^n_g(q) \in \frac{1}{(t_1t_2)^n}H^*(\Mbar_{g,k})[t_1,t_2]_m[\![q]\!] \subset H^*(\Mbar_{g,k})[t^\pm_1,t^\pm_2][\![q]\!]. 
 \]
 More explicitly, 
\[
\mathsf{GW}(\mu^1, \hdots, \mu^k)^n_g(q)=\frac{1}{(t_1t_2)^n} \sum_{i < \infty} A_ip_i(t_1,t_2)g_i(q),
\]
where $A_i \in H^*(\Mbar_{g,k})$, $p_i(t_1,t_2)$ are polynomials, and $g_i(q)$ are some $q$-series. By (\ref{rat}), the total series is a rational function.  This implies that $g_i(q)$ must be rational functions themselves, 
\[
\mathsf{GW}(\mu^1, \hdots, \mu^k)^n_g(q) \in  H^*(\Mbar_{g,k})[t^\pm_1,t^\pm_2] \otimes_{\BQ} \BQ(q), 
\]
as desired. If $k=0$, we introduce a marking and insert a divisor. 
\bibliographystyle{amsalpha}
\bibliography{CRC}


\end{document}